\documentclass[final]{elsarticle}

\usepackage{amsmath}
\usepackage{amsthm}
\usepackage{amssymb}
\usepackage{amsfonts,dsfont}
\usepackage{graphicx,xcolor}
\usepackage{array}
\usepackage{arydshln,MnSymbol}
\usepackage{MnSymbol}
\usepackage{dashrule}
\usepackage{enumitem}
\usepackage{hyperref,bookmark}
\hypersetup{
citebordercolor= {1 0 0},
linkbordercolor={0 0 1},
linktoc={all},
}
\setlistdepth{19}

\setlist[itemize,1]{label=$\bullet$}
\setlist[itemize,2]{label=$\bullet$}
\setlist[itemize,3]{label=$\bullet$}
\setlist[itemize,4]{label=$\bullet$}
\setlist[itemize,5]{label=$\bullet$}
\setlist[itemize,6]{label=$\bullet$}
\setlist[itemize,7]{label=$\bullet$}
\setlist[itemize,8]{label=$\bullet$}
\setlist[itemize,9]{label=$\bullet$}
\setlist[itemize,10]{label=$\bullet$}
\setlist[itemize,11]{label=$\bullet$}
\setlist[itemize,12]{label=$\bullet$}
\setlist[itemize,13]{label=$\bullet$}
\setlist[itemize,14]{label=$\bullet$}
\setlist[itemize,15]{label=$\bullet$}
\setlist[itemize,16]{label=$\bullet$}
\setlist[itemize,17]{label=$\bullet$}
\setlist[itemize,18]{label=$\bullet$}

\renewlist{itemize}{itemize}{19}

\usepackage{subcaption}
\usepackage{mdframed}

\newcommand{\scalar}[2]{\langle#1\,,#2\rangle}
\newcommand{\scalarb}[2]{\langle#1\,,#2\rangle_{\partial\Omega}}
\newcommand{\norm}[1]{\|#1\|}
\newcommand{\normm}[1]{|\negthinspace\|#1\|\negthinspace|}

\renewcommand{\H}{\mathcal{H}}
\newcommand{\D}{\mathcal{D}}
\renewcommand{\L}{\mathcal{L}}
\renewcommand{\d}{\mathrm{d}}
\newcommand{\pO}{{\partial \Omega}}
\newcommand{\C}{\mathcal{C}}
\newcommand{\1}{\mathbb{I}}

\newcommand{\diagdots}[3][-25]{%
  \rotatebox{#1}{\makebox[0pt]{\makebox[#2]{\xleaders\hbox{$\cdot$\hskip#3}\hfill\kern0pt}}}%
}

\definecolor{mygreen}{rgb}{0, 0.7, 0.4}
\definecolor{myviolet}{rgb}{0.5, 0, 1}
\definecolor{line1}{rgb}{0.85, 0,0.9 }
\definecolor{line2}{rgb}{0, 0,0.55 }
\definecolor{line3}{rgb}{0, 0,1 }
\definecolor{line4}{rgb}{0.44, 0,1 }
\definecolor{line5}{rgb}{0.2, 0.6,1 }

\newmdenv[
  topline=false,
  bottomline=false,
  rightline=false,
  skipabove=\topsep,
  skipbelow=\topsep,
  linewidth=0.8pt,
  rightmargin=-0.2cm
]{siderules}

\newtheorem{theorem}{Theorem}
\newtheorem{proposition}[theorem]{Proposition}

\newtheorem{lemma}[theorem]{Lemma}
\newtheorem{definition}[theorem]{Definition}

\newtheorem{assumption}{Assumption}

\numberwithin{equation}{section}
\numberwithin{theorem}{section}

\makeatletter
\def\ps@pprintTitle{%
 \let\@oddhead\@empty
 \let\@evenhead\@empty
 \def\@oddfoot{}%
 \let\@evenfoot\@oddfoot}
\makeatother

\begin{document}

\begin{frontmatter}

\title{Finite element method to solve the spectral problem for arbitrary self-adjoint extensions of the Laplace-Beltrami operator on manifolds with a boundary\tnoteref{titlenote}}
\tnotetext[titlenote]{The authors would like to give special thanks to Prof.\ A.\ Ibort for his valuable comments and useful discussions. They also acknowledge the partial support by the Spanish MINECO grant MTM2014-54692-P and QUITEMAD+, S2013/ICE-2801. A.\ L-Y.\ has been partially supported  by NICOP, N62909-15-1-2011.}

\author[address1]{A.\ L\'opez-Yela}
\ead{alyela@tsc.uc3m.es}

\author[address2,address3]{J.M.\ P\'erez-Pardo\corref{corr}}
\cortext[corr]{Corresponding author}
\ead{jmanuelperpar@gmail.com, jmppardo@math.uc3m.es}

\address[address1]{Depto.\ de teor\'ia de la se\~nal y telecomunicaciones, Univ.\ Carlos III de Madrid, Avda.\ de la Universidad 30, 28911 Legan\'es, Madrid, Spain.}
\address[address2]{INFN, Sezione di Napoli, Complesso Universitario di Monte S.~Angelo, via Cintia, 80126 Naples, Italy.}
\address[address3]{Depto.\ de Matem\'aticas, Univ.\ Carlos III de Madrid, Avda.\ de la Universidad 30, 28911 Legan\'es, Madrid, Spain.}

\begin{abstract}
A numerical scheme to compute the spectrum of a large class of self-adjoint extensions of the Laplace-Beltrami operator on manifolds with boundary in any dimension is presented. The algorithm is based on the characterisation of a large class of self-adjoint extensions of Laplace-Beltrami operators in terms of their associated quadratic forms. The convergence of the scheme is proved. A two-dimensional version of the algorithm is implemented effectively and several numerical examples are computed showing that the algorithm treats in a unified way a wide variety of boundary conditions.
\end{abstract}


\end{frontmatter}


\section{Introduction}\label{sec:Intro}
The study of self-adjoint extensions of symmetric operators plays a fundamental role not only in the foundations, but increasingly so in the applications of Quantum Mechanics as they determine the spectrum of the corresponding system.  Among them, it is paramount the role played by the Laplace-Beltrami operator, as it corresponds to the to the time independent Schr\"odinger equation for a free particle.

When boundaries are present such operators can be defined easily on a domain where it is symmetric but usually not self-adjoint. Self-adjointness is a crucial property that guarantees the reality of the spectrum. Moreover, Stone's Theorem establishes that it also guarantees the unitarity of the evolution governed by the Schr\"odinger equation. See, e.g. \cite[Chapter X]{reed75} for further details and motivation. Starting from a symmetric operator one needs to choose a self-adjoint extension of it, which in general is not unique. In the present context of differential operators on manifolds with boundaries this is done by selecting appropriately boundary conditions. Different boundary conditions represent different physical situations, see for instance the reviews \cite{As15,Ib15} and references therein. Consider the Laplace operator on an interval. Dirichlet boundary conditions represent a particle trapped in a box while quasiperiodic boundary conditions represent that the particle is moving on a closed curve surrounding a magnetic field \cite{As83}. Other boundary conditions represent other physical situations. For instance, they can be chosen to represent point like interactions \cite{DAFT08}.

In dimension one, the problem of characterising self-adjoint extensions and computing their spectrum and eigenvectors was addressed in \cite{Ib13}. However, the algorithm proposed there cannot be applied in dimension higher than one in a straightforward way. The main reason for this is that in dimension higher than one the space of self-adjoint extensions is infinite dimensional \cite{Grubb68,ibortlledo14b}. In dimension one, the self-adjoint extensions can be parameterised by the set of unitary operators acting on a finite dimensional vector space and therefore they can be implemented exactly. This is not the case in dimension higher than one where in general one needs also to approximate the boundary conditions. This needs to be handled carefully when proving the convergence of a numerical scheme approximating the spectrum. The numerical study of self-adjoint extensions for the Laplace-Beltrami operator in dimension higher than one is also interesting from a mathematical perspective since it requires the development of finite element methods (FEM) that use completely different constructions of boundary elements than for the already well developed Dirichlet and Neumann boundary conditions, cf. \cite{babuska91, boffi10, ram02}.

New quantum technologies and applications require the implementation of boundary conditions that go beyond the usual ones in order to have a good description of their properties. For instance, Quantum Hall effect \cite{morandi88, BZ06}, superconductors surrounded by insulators \cite{asorey13, asorey16}, Casimir Effect \cite{plunien86, marolf97, asorey06}, computation of solutions of Bloch periodic wave-functions on periodic lattices \cite{sukumar09} and other novel proposals like the generation of entanglement or the study of topology change by modifying the boundary conditions \cite{IMP, Pe15}. 
Self-adjoint boundary conditions can also be used to model physical situations like point interactions \cite{DAFT08} or resonators coupled to thin antennas \cite{CP17, ES97, ES07}.
It is important to notice that the addition of regular potentials does not jeopardise the self-adjointness of the domain of a differential operator, cf.\ \cite{reed75,kato95}. Hence, the analysis carried out in this article can be used straightforwardly for Schr\"odinger operators by just computing the contribution of the potential as it is done in standard FEM. 

In this context, boundary conditions are going to be treated as the only input parameter of the problem and the geometry will remain fixed. In the standard FEM approach, one needs to distinguish \emph{a priori} which boundary conditions are essential and which ones are natural in order to construct the appropriate FEM. In contrast, the algorithm presented here treats natural and essential boundary conditions in a unified way. It is able to deal with a diversity of boundary conditions like Dirichlet, Neumann, Robin, mixed, periodic, quasi-periodic (also called Bloch-periodic), or even more general ones like those appearing in \cite{Ib14} by just modifying the input parameters.

This article focuses on the construction and analysis of the aforementioned algorithm to show its capabilities. Moreover, the convergence conditions of this approach are proven not only for the particular realisation presented here, but for a general situation. The approach for the construction of finite elements at the boundary, as it is proposed here, should be taken as a complement to already existent and well-established all-purpose routines, for instance, the one presented in \cite{strouboulis00}. However, standard approaches do not allow to solve the problem for the variety of boundary conditions that can be handled with the present scheme. Implementation of more efficient approaches to speed up convergence will be considered in the future. These include mesh refinements or increasing the polynomial degree of the finite elements, i.e.\ implementation of h-method or p-method, as well as other recent developments like including probabilistic indetermination of the input data as it is done in \cite{babuska07,babuska04,babuska05}. These latter considerations will play a relevant role when considering more complex geometries than the ones considered here.

The implementation of the FEM to cope with boundary conditions defined by unitary operators at the boundary is performed by adding a rim of boundary elements to the domain that serve to implement the finite dimensional approximation of the domain of the given operator.  Such elements have a particular structure that has been carefully crafted to guarantee the convergence of the domains and the quadratic forms approximating the original problem. In order to implement all these different boundary conditions it is only needed to add the aforementioned rim of boundary elements. In the interior of the domain, the bulk, one can use well developed numerical schemes that are already available, for instance \cite{strouboulis00}. 

The article is organised as follows. In Section \ref{sec:QF} we introduce the family of self-adjoint extensions of the Laplace-Beltrami operator that is suitable for the numerical approximation of its spectrum. In Section \ref{sec:ArbitraryDimension} we provide sufficient conditions on the approximants of the spectral problem to guarantee convergence to the exact solutions and in Section \ref{sec:FEM} we construct them explicitly. In order to test the performance of the scheme we have built a standard finite element method at the bulk and we have complemented it with the proposed construction at the boundary. In Section \ref{sec:examples} several numerical experiments with applications in Physics have been solved to show the capabilities of the proposed scheme. The pseudocode of the implementation can be found in the appendix.


\section{Self-adjoint extensions of the Laplace-Beltrami operator and unitary operators at the boundary}\label{sec:QF}

In this section we introduce the family of operators that will be addressed by the numerical algorithm. We will present the most important results in order to keep the article as self-contained as possible. This will also serve to fix the notation.

Let $(\Omega,\pO,\eta)$ be a smooth orientable Riemannian manifold with metric
$\eta$ and smooth compact boundary $\partial \Omega$.
We will denote as  $\C^\infty (\Omega)$ the space of smooth functions on the Riemannian manifold $\Omega$, and
by $\C_c^\infty (\Omega)$ the space of smooth functions with compact support in the interior of $\Omega$.
The Riemannian volume form is denoted by $\d\mu_\eta$.

The \textbf{Laplace-Beltrami Operator} associated to the Riemannain manifold $(\Omega,\pO,\eta)$ is the
second order differential operator $\Delta_\eta:\C^\infty(\Omega)\to\C^\infty(\Omega)$ given in local coordinates $(x^i)$ on $\Omega$ by
$$\Delta_\eta\Phi=\sum_{i,j}\frac{1}{\sqrt{|\eta|}}\frac{\partial}{\partial x^i}\sqrt{|\eta|}\eta^{ij}\frac{\partial\Phi}{\partial x^j}\;.$$

The different self-adjoint extensions of this operator describe the dynamics of a quantum free particle moving on the manifold $\Omega$\,. 
The Laplace-Beltrami operator can be defined in the domain consisting of smooth functions $\C^\infty (\Omega)$ on $\Omega$, but it is not self-adjoint on it (neither symmetric). If the domain is restricted to the subspace $\C_c^\infty (\Omega)$, the operator is symmetric but not self-adjoint. It is not obvious, in fact it is a difficult problem, how to choose the domains appropriately in such a way to make the operator (essentially) self-adjoint, cf.\ \cite{Grubb68,ibortlledo14b}. Moreover, once this is done, it becomes a problem to compute efficiently its spectrum .

In what follows, we describe a class of weak problems that are associated to a large family of self-adjoint extensions of the Laplace-Beltrami operator. This class contains all the well-known boundary conditions that lead to self-adjoint extensions of the Laplace-Beltrami operator like Dirichlet, Neumann or periodic boundary conditions. Furthermore, this class allows to numerically approximate its spectrum, as is showed in Section~\ref{sec:ArbitraryDimension}.

The \textbf{Sobolev space of order} $k$ on the manifold $\Omega$ will be denoted by $\H^k(\Omega)$ and the associated norm and scalar product by $\norm{\cdot}_k$ and $\scalar{\cdot}{\cdot}_k$ respectively. When the manifold to be considered is not clear from the context, we will use the more explicit subindices $\norm{\cdot}_{\H^k(\Omega)}$ and $\scalar{\cdot}{\cdot}_{\H^k(\Omega)}$\,.
The boundary of the Riemannian manifold is going to be denoted by $\pO$ and the Riemannian measures, on the manifold $\Omega$ and its boundary $\pO$, by $\mu_{\eta}$ and $\mu_{\partial\eta}$ respectively. 
The spaces of smooth functions over the two manifolds verify that $\C^\infty(\Omega)\bigr|_{\pO}\simeq\C^{\infty}(\pO)$. We will need the following results.

\begin{definition} \label{def:tracemap}
Let $\gamma: \H^1(\Omega) \to \H^{1/2}(\pO)$ be the surjective map that provides the restriction to the boundary $\gamma(\Phi) = \Phi|_{\pO}$\,, cf.\ \cite{lions72}. Let $\H^1_0 := \ker \gamma$. The \textbf{boundary Sobolev space} $\H_b$ is the orthogonal complement of $\H^1_0$ with respect to the scalar product in $\H^1(\Omega)$. That is 
$$\H^1(\Omega) = \H_b\oplus\H^1_0\,.$$
The orthogonal projections onto these subspaces are denoted respectively by $\pi_b$ and $\pi_0$.
\end{definition}

A direct consequence of \cite[Theorem 8.3]{lions72}, \cite[Theorem 7.39]{adams03} is the following:

\begin{proposition} \label{prop:boundarymap}
	Let $\gamma_b: \H_b \to \H^{1/2}(\pO)$ be the restriction of the trace map to the boundary Sobolev space. Then $\gamma_b$ is a continuous bijection.
\end{proposition}

In order to prove the convergence results we will use the well-known notion of closed, semibounded quadratic forms and their relation with self-adjoint operators. These provide the appropriate analytic setting for our purposes. Standard references on this subject are \cite[Chapter~VI]{kato95}, \cite[Section~VIII.6]{reed72} or \cite[Section~4.4]{davies95}. We recall here some results from \cite{ibortlledo14b} that describe a large class of self-adjoint extensions of the Laplace-Beltrami operator. 
The extensions are parameterized in terms of suitable unitary operators on the Hilbert space of square integrable functions on the boundary. 

\begin{definition}\label{DefGap}
Let $U$ be unitary and denote its spectrum by $\sigma(U)$. The unitary operator $U$ \textbf{has spectral gap at $-1$} if one of the following conditions hold:
\begin{enumerate}
\item $\1+U$ is invertible.
\item $-1\in\sigma(U)$ and $-1$ is not an accumulation point of $\sigma(U)$.
\end{enumerate}
The eigenspace associated to the eigenvalue $-1$ is denoted by $W$ and called the \textbf{relevant subspace}. The corresponding orthogonal projections will be written as $P_U$ and $P_{U}^\bot=\1-P_U$.
\end{definition}

\begin{definition}\label{partialCayley}
Let $U$ be a unitary operator with spectral gap at $-1$\,. 
The \textbf{partial Cayley transform} $A_U\colon\L^2(\pO)\to \L^2(\pO)$ is
the bounded operator
\[
A_U:=i\,P_{U}^\bot (U-\mathbb{I}) (U+\mathbb{I})^{-1}\;.
\]
\end{definition}

\begin{definition}\label{def:admissible}
Let $U: \L^2(\pO)\to\L^2(\pO)$ be a unitary operator with spectral gap at $-1$\,. The unitary is said to be \textbf{admissible} if the partial Cayley transform $A_U$ leaves the subspace $\H^{1/2}(\pO)$ invariant and is continuous with respect to the Sobolev norm of order $1/2$\,, i.e. there exists $K\in\mathbb{R}$ such that
$$\norm{A_U\varphi}_{\H^{1/2}(\pO)}\leq K \norm{\varphi}_{\H^{1/2}(\pO)}\;,\quad \forall \varphi \in \H^{1/2}(\pO)\;.$$
\end{definition}

\begin{definition}
Let $U: \L^2(\pO)\to\L^2(\pO)$ be a unitary operator with spectral gap at $-1$\,. The domain $\D_U$ associated to the unitary operator $U\in \mathcal{U}\left(L^2(\pO)\right)$ with gap at $-1$ is defined by
\begin{equation}\label{domaincha4}
\D_U=\bigl\{ \Phi\in\H^1(\Omega)\bigr|   \; P_U\gamma(\Phi)=0 \bigr\}\;.
\end{equation}
\end{definition}

\begin{definition}\label{DefQU}
Let $U: \L^2(\pO)\to\L^2(\pO)$ be a unitary operator with spectral gap at $-1$\,, $A_U$ the corresponding partial Cayley transform and $\gamma$
the trace map considered in Def.~\ref{def:tracemap}.
The Hermitean quadratic form $Q_U$ with domain $\D_U$ is defined by
$$Q_U(\Phi,\Psi)=\scalar{\d\Phi}{\d\Psi}-\scalarb{\gamma(\Phi)}{A_U\gamma(\Phi)}\;.$$
\end{definition}

The next theorem characterises the class of self-adjoint extensions of the minimal Laplace-Beltrami operator $-\Delta_{\mathrm{min}}$ that we will be interested in. We refer to \cite{ibortlledo14b} for a complete proof and additional motivation.\\

\begin{theorem}\label{thm:char_U}
Let $U\colon\L^2(\pO)\to\L^2(\pO)$ be an admissible unitary operator with spectral gap at $-1$. 
Then, the quadratic form $Q_U$ with domain $\D_U$ is semi-bounded from below and closable. Its closure 
is represented by a semi-bounded self-adjoint extension of the minimal Laplacian $-\Delta_{\mathrm{min}}$\,. We shall denote this self-adjoint extension by $-\Delta_U$. 
\end{theorem}

The relation between the unitary operator $U$ and the self-adjoint operator $-\Delta_U$ can be summarised as follows. The self-adjoint operator $-\Delta_U$ is the self-adjoint extension of the minimal Laplace-Beltrami operator that satisfies the boundary condition
\begin{equation}
	\varphi-i\dot{\varphi} = U(\varphi + i\dot{\varphi})\;.
\end{equation}
Here, $\varphi = \gamma(\Phi)$ and $\dot{\varphi} = \gamma\left( \frac{\d\Phi}{\d\mathbf{n}} \right)$, i.e.\ the restriction to the boundary of the function $\Phi$ and its normal derivative respectively. The normal direction is taken pointing outwards the manifold. More specifically, $-\Delta_U$ is the self-adjoint extension with domain
\begin{equation}\label{eq:asorey}
	\D(\Delta_U) = \{\Phi \in \H^2(\Omega) \mid \varphi-i\dot{\varphi} = U(\varphi + i\dot{\varphi}) \}\;.
\end{equation}


\section{Approximations of the spectral problem in arbitrary dimension by finite element methods}\label{sec:ArbitraryDimension}

In this section we develop a class of numerical algorithms, based in the finite element method, that can be used to approximate the spectral problem for the self-adjoint extensions of the Laplace-Beltrami operator described in Section \ref{sec:QF}. A standard reference for this method is \cite{brenner08}.

In what follows, let $-\Delta_U$ denote the self-adjoint operator associated to the closure of the quadratic form $Q_U$ of Def.~\ref{DefQU}. Thm.~\ref{thm:char_U} ensures that this closure exists and that the operator $-\Delta_U$, with domain $\D(\Delta_U)$\,, is semi-bounded from below. Moreover, $-\Delta_U$ is a self-adjoint extension of $-\Delta_\mathrm{min}$.

We are interested in obtaining numerical approximations of pairs $(\Phi,\lambda)\in\L^2(\Omega)\times\mathbb{R}$ that are solutions of the spectral problem
\begin{equation}\label{spectralproblem}
	-\Delta_U\Phi=\lambda\Phi\quad\Phi\in\D(\Delta_U)\;.
\end{equation}
As usual in finite element methods, the solutions can be obtained by approximating the solution of the problem in the weak form. That is, a pair $(\Phi,\lambda)\in\L^2(\Omega)\times\mathbb{R}$ is a solution of the spectral problem \eqref{spectralproblem} if and only if it is a solution of the weak spectral problem
	\begin{equation}\label{weakspectralproblem}
		Q_U(\Psi,\Phi)=\lambda\scalar{\Psi}{\Phi}\quad\forall\Psi\in\D_U\;,
	\end{equation}
with $\D_U$ the domain of the quadratic form $Q_U$ given in Eq.~\eqref{domaincha4}. Proving this equivalence is straightforward using the definitions of the previous section.

The way we shall approximate the solution of the weak spectral problem \eqref{weakspectralproblem} is by finding a family of finite-dimensional problems, in terms of an appropriate family of finite-dimensional spaces $\{S^N_{U}\}$, that verifies approximately the boundary conditions. Then, we will look for solutions $(\Phi_N,\lambda_N)\in S^N_{U}\times\mathbb{R}$ of these approximate spectral problems. Since the boundary conditions, represented by the unitary operator $U$, will also be approximated, we need to look for solutions of a finite-dimensional problem of the form
	\begin{equation}\label{approximatespectralproblem}
		Q_{U^N}(\Psi_N,\Phi_N)=\lambda_N\scalar{\Psi_N}{\Phi_N}\, ,\qquad\forall \Psi_N\in S^N_{U}\;.
	\end{equation}
where $\{U^N\}$ is a family of unitary operators.

The rest of this section is devoted to obtain sufficient conditions on the family $\{S^N_{U}\}$ such that the solutions of the approximate spectral problem \eqref{approximatespectralproblem} converge to the solutions of the weak spectral problem \eqref{weakspectralproblem}. In Section \ref{sec:FEM}, we construct explicitly a family $\{S^N_{U}\}$ for the two-dimensional case. Due to the non-locality of generic boundary conditions, one needs to introduce a non-local subspace of functions that is able to encode them.

Before introducing the actual algorithm, which will be done in Section~\ref{sec:FEM}, we discuss the general conditions guaranteeing convergence of the proposed numerical scheme. In particular, one needs to address the problem of approximating the boundary conditions.
It is worth mentioning that the considerations of this section work regardless of the dimension of the underlying manifold and serve as a stepping stone for implementations in any dimension.

In what follows we assume that the unitary operator $U$ describing the quadratic form $Q_U$ is admissible, cf.\ Def.~\ref{def:admissible}. Hence, we are under the conditions of Thm.~\ref{thm:char_U} and therefore the quadratic form
\begin{equation}\label{QFcha4}
	Q_U(\Phi,\Psi)=\scalar{\d\Phi}{\d\Psi}-\scalarb{\gamma(\Phi)}{A_U\gamma(\Phi)}
\end{equation}
with domain $\D_U=\bigl\{ \Phi\in\H^1(\Omega)\bigr|P_U\gamma(\Phi)=0 \bigr\}$\,, 
where $A_U$ is the partial Cayley transform of Def.~\ref{partialCayley}, is closable and semi-bounded from below.\\

We introduce now sufficient conditions on the family of finite-dimensional problems to ensure convergence to the solutions of the aforementioned spectral problem \eqref{spectralproblem}. First of all, one needs a family of finite-dimensional subspaces that approximate the Sobolev spaces of order 1. We shall call each member of this family a \textbf{finite elements space} and denote it by $S^N$, where $N\in\mathbb{Z}$ denotes the vector space dimension. We denote by $\{\Pi^N\}_{N\in\mathbb{Z}}$ the family of projections onto these subspaces: 
$$\Pi^N:\H^1(\Omega)\to S^N\;.$$
In general, this projections are not orthogonal projections.
As a general assumption we need to impose that 
\begin{equation}\label{eq:h1aprox}
	\norm{(\Pi^N-\1)\Phi}_1\stackrel{N\to\infty}{\longrightarrow}0\;.
\end{equation}
Since each $S^N$ is a closed subspace of $\H^1(\Omega)$, one can define the traces of these subspaces as
\begin{equation}
	s^N=\{\varphi\in\H^{1/2}(\pO) | \varphi=\gamma(\Phi)\;,\;\Phi\in S^N\}\;.
\end{equation}
We will call them \textbf{traces of the finite element spaces}. Notice that the traces of the finite element spaces do not have dimension $N$ in general, but are also finite-dimensional subspaces with $\operatorname{dim}{s^N}\leq N$.\\

\begin{assumption}\label{ass:1}
Let $U\in\mathcal{U}(\L^2(\pO))$ be an admissible unitary operator with gap. Let $\{U^N\}_{N\in \mathbb{Z}}$ be a family of unitary operators such that $U^N\in\mathcal{U}(s^N)$. The projections onto the relevant subspaces $$P_{U^N}:s^N\to s^N$$ and the partial Cayley transforms $$A_{U^N}:s^N\to s^N$$  can be lifted to $\H^{1/2}(\pO)$ by considering that the orthogonal complement to $s^N$ with respect to the scalar product in $\H^{1/2}(\pO)$ is in their kernel. Let the operators constructed this way be denoted with the same symbol respectively. We assume:
	\begin{enumerate}
		\item{ $\norm{(P_{U^N}-P_U)\varphi}_{\H^{1/2}(\pO)} \stackrel{N\to\infty}{\longrightarrow} 0\;,  \quad \varphi\in\H^{1/2} $}\;.\label{en:ass1i}\\
		\item {$\forall \Psi \in \L^2(\pO), \quad \scalar{\psi}{ (A_{U^N}-A_U)\varphi}_{\L^2(\pO)} \stackrel{N\to\infty}{\longrightarrow} 0\;, \quad \varphi\in\H^{1/2}$}\;.\label{en:ass1ii}\\
	\end{enumerate}
\end{assumption}

\begin{definition}\label{def: approx family}
	Let $\{U^N\}_N$ be a family of unitary operators satisfying Assumption~\ref{ass:1}. Consider the finite-dimensional subspaces
	\begin{equation}
		S^N_{U}=\{\Phi\in S^N| P_{U^N}\gamma(\Phi)=0\}\,.
	\end{equation}
	The family $\{S^N_U\}_N$ will be called an \textbf{approximating family} of the spectral problem \eqref{spectralproblem}.\\
\end{definition}

We prove the convergence in two steps. First, we shall prove the convergence of the eigenvalues. After that we will be able to prove the convergence of the eigenfunctions.\\

\begin{lemma}\label{lem:conveigen}
	Let $\{S^N_U\}_N$ be an approximating family and let $\{\lambda^N)\}_N$ be the sequence of eigenvalues corresponding to the $n$-th lowest eigenvalue of the approximate spectral problem \eqref{approximatespectralproblem}. This sequence converges to the n-th lowest eigenvalue of the weak spectral problem \eqref{weakspectralproblem}.
\end{lemma}

\begin{proof}
	Let $V_n$ and $V^N_n$ be the spaces
	\begin{align}
	    V_n=\{\Phi\in\D_U\bigr|\scalar{\Phi}{\xi_i}=0\,,\xi_i\in\L^2(\Omega)\,,i=1,\dots,n\}\;,\phantom{\sum_{N}}\\
	    V^N_n=\{\Phi_N\in S^N_U \bigr|\scalar{\Phi_N}{\xi_i}=0\,,\xi_i\in\L^2(\Omega)\,,i=1,\dots,n\}\;.
	\end{align}
Then, applying the \textrm{min-max} Principle, cf. \cite[Thm.~XIII.2]{reed78}, to the quadratic form $Q_U$ of Def.~\ref{DefQU} on the domains of the weak and the approximate spectral problems, and subtracting them, we get:
\begin{equation}\label{eq:minmax}
\sup_{\xi_1,...,\xi_{n-1}}\left[\inf_{\Phi_N\in V^N_{n-1}} \frac{Q_U(\Phi _N)}{\norm{\Phi _N}^2}-\inf_{\Phi\in V_{n-1}} \frac{Q_U(\Phi)}{\norm{\Phi}^2}\right]=\widetilde{\lambda^N}-\lambda\;.
\end{equation}
Where $\widetilde{\lambda^N}$ is the $n$-th eigenvalue of the finite-dimensional problem defined by the quadratic form $Q_U$ with domain $S^N_U$ and $\lambda$ is the $n$-th eigenvalue of the weak spectral problem. Notice that $\widetilde{\lambda^N}$ is not an eigenvalue of the approximate problem \eqref{approximatespectralproblem} since for the moment we are considering only the quadratic form $Q_U$ and not $Q_{U^N}$. Let $\normm{\cdot}^2:=Q_U(\cdot) + (m+1)\norm{\cdot}^2$ be the graph-norm of the quadratic form $Q_U$, where $m$ is its lower bound, cf.\ Thm.~\ref{thm:char_U}. By adding and subtracting $(m+1)\norm{\Phi}^2$, forgetting the supremum for the moment, and considering that all the functions are normalised in $\L^2(\Omega)$, we get that the left hand side of Eq.~\eqref{eq:minmax} verifies for $\Phi\in V_{n-1}\subset\D_U$:
\begin{multline}
	\inf_{\Phi_N\in V^N_{n-1}}\normm{\Phi_N}^2-\inf_{\tilde{\Phi}\in V_{n-1}}\normm{\tilde{\Phi}}^2
		\leq\normm{\Phi}^2 -\inf_{\tilde{\Phi}\in V_{n-1}}\normm{\tilde{\Phi}}^2
		\\+ \inf_{\Phi_N\in V^N_{n-1}}\left( \normm{\Phi-\Phi_N}^2 + 2\normm{\Phi}\normm{\Phi-\Phi_N} \right)\,,
\end{multline}
Now, for all $\epsilon>0$ we can select $\Phi\in V_{n-1}$ such that
$$\left| \normm{\Phi}^2 -\inf_{\tilde{\Phi}\in V_{n-1}}\normm{\tilde{\Phi}}^2\right| <\epsilon\;.$$ Therefore,

\begin{equation}
	\inf_{\Phi_N\in V^N_{n-1}}\normm{\Phi_N}^2-\inf_{\tilde{\Phi}\in V_{n-1}}\normm{\tilde{\Phi}}^2
		\leq \epsilon + \inf_{\Phi_N\in V^N_{n-1}}\left( \normm{\Phi-\Phi_N}^2 + 2\normm{\Phi}\normm{\Phi-\Phi_N} \right)\;.
\end{equation}
Since the graph norm of the quadratic form is dominated by the Sobolev norm $\norm{\cdot}_1$, in order to bound the right hand side it is enough to show that for a fixed $\Phi\in\D_U$
$$\inf_{\Phi_N\in S^N_U}\norm{\Phi-\Phi_N}_1\stackrel{N\to\infty}{\longrightarrow}0\;.$$
From Eq.~\eqref{eq:h1aprox}, there exists for every $\Phi\in\D_U\subset\H^1$ and every $\epsilon>0$ an $N_0$ such that for all $N>N_0$ we have that
\begin{equation}\label{eq:conv.eign.aprox.h1}
	\norm{\Phi-\Pi^N\Phi}_{1}^2\leq\epsilon\;.
\end{equation}
Notice that $\Pi^N\Phi$ is not an element of $S^N_U$ in general\,. Using the decomposition of Def.~\ref{def:tracemap} and Prop.~\ref{prop:boundarymap} we can define
$$\widetilde{\Phi^N}:=\gamma^{-1}_b(\1-P_{U^N})\gamma_b\pi_b\Pi^N\Phi+\pi_0\Pi^N\Phi\;.$$
It is easy to check that $\widetilde{\Phi^N} \in S^N_U$. Now, we get: 
\begin{align*}
	\norm{\Phi-\widetilde{\Phi^N}}_1^2&=\norm{\pi_0\Phi-\pi_0\widetilde{\Phi^N}}_1^2+\norm{\pi_b\gamma^{-1}_b(\1-P_{U^N})\gamma_b\pi_b\Pi^N\Phi-\pi_b\Phi}_1^2\\
		&\leq \norm{\pi_0(\1-\Pi^N)\Phi}_1^2 + K\norm{(\1-P_{U^N})\gamma_b\pi_b\Pi^N\Phi-\gamma(\pi_b\Phi)}_{\H^{1/2}(\pO)}^2\\
		&\leq \norm{\pi_0(\1-\Pi^N)\Phi}_1^2 + K\norm{(\1-P_{U^N})\left(\gamma_b\pi_b\Pi^N\Phi-\gamma_b\pi_b\Phi\right)}_{\H^{1/2}(\pO)}^2\\
		&\phantom{aaaaaaaaaaa} + K\norm{(P_{U^N}-P_{U})\gamma_b\pi_b\Phi}_{\H^{1/2}(\pO)}^2\;,
\end{align*}
where $K$ is the Lipschitz constant of the bijection $\gamma_b$. The right hand side can be chosen as small as needed by means of Eq.~\eqref{eq:conv.eign.aprox.h1} and Asumption~\ref{ass:1}. Similar arguments lead to a similar bound for the reverse difference in Eq.~\eqref{eq:minmax}, and hence we have proved that $$|\widetilde{\lambda_N}-\lambda|\to0\;.$$

Finally, we need to show that the $n$-th eigenvalue of the finite-dimensional problems $Q_U$ and $Q_{U^N}$, both with domain $S^N_U$ converge, i.e.\ $|\lambda_N-\widetilde{\lambda_N}|\to0$\;. We shall use the min-max principle again. Notice that
\begin{equation}
	\frac{Q_U(\Phi)}{\norm{\Phi}^2} - \frac{Q_{U^N}(\Phi)}{\norm{\Phi}^2}=-\frac{\scalar{\varphi}{(A_U-A_{U^N})\varphi}}{\norm{\Phi}^2}\leq K\norm{A_U-A_{U^N}}_{L(S_U^N)}\frac{\norm{\Phi}^2_1}{\norm{\Phi}^2}\;.
\end{equation}
If we apply the min-max Principle to both sides, we get that 
\begin{equation}\label{eq:convergenceeigenvalues}
	|\widetilde{\lambda_N}-\lambda_N| \leq K'\norm{A_U-A_{U^N}}_{L(S_U^N)}\sup_{\xi_1,...,\xi_{n-1}}\inf_{\Phi_N\in V^N_{n-1}}\frac{\norm{\Phi}^2_1}{\norm{\Phi}^2}\;.
\end{equation}
Now notice that $\norm{\Phi}_{1}^2$ is proportional to the quadratic form associated to the Dirichlet extension of the Laplace-Beltrami operator. Since $S^N_{\mathrm{Dirichlet}}\subset S_U^N$, we have that
$$\sup_{\xi_1,...,\xi_{n-1}}\inf_{\Phi_N\in V^N_{n-1}}\frac{\norm{\Phi}^2_1}{\norm{\Phi}^2}\leq K(\lambda^D+1)\;,$$
where $\lambda^D$ is the $n$-th eigenvalue of the Dirichlet extension of the Laplace-Beltrami operator on the manifold $\Omega$\,. Finally, notice that the factor $\norm{A_U-A_{U^N}}_{L(S_U^N)}$ is the operator norm over a finite-dimensional subspace of $\L^2(\pO)$\,. Hence, the weak convergence condition in Assumption \ref{ass:1} guarantees that the right hand side of Eq. \eqref{eq:convergenceeigenvalues} can be chosen as small as needed.
\end{proof}
\medskip

\begin{lemma}
	Let $\{S_U^N\}_N$ be an approximating family and let $\{\Phi^N\}_N$ be a sequence of eigenfunctions associated to the sequence $\{\lambda_N\}_N$ of Lemma~\ref{lem:conveigen}. Let $\normm{\cdot}$ be the graph norm of the quadratic form $Q_U$. Then, by passing to a subsequence if necessary, there exists $\Phi\in\overline{\D_U}^{\normm{\cdot}}$ such that for any $\Psi\in\overline{\D_U}^{\normm{\cdot}}$
	$$Q_U(\Psi,\Phi-\Phi^N)\stackrel{N\to\infty}{\longrightarrow}0\;,$$ and such that for all $\Psi\in\overline{\D_U}^{\normm{\cdot}}$
	$$Q(\Psi,\Phi)=\lambda\scalar{\Psi}{\Phi}\;.$$
\end{lemma}

\begin{proof}
	We can assume that $\norm{\Phi^N}_1=1$\,. By the Banach-Alaoglu theorem and the fact that the Sobolev norm of order 1 dominates the graph-norm of the quadratic form, there exist subsequences $\{\Phi^{N_j}\}$ and accumulations points $\Phi$ such that for any $\Psi\in\overline{\D_U}^{\normm{\cdot}}$ it holds that
	$$Q_U(\Psi,\Phi-\Phi^N_j)\stackrel{N_j\to\infty}{\longrightarrow}0\;.$$
	Proving that $\Phi$ is a solution of the weak eigenvalue problem is straightforward using that weak convergence in $\overline{\D_U}^{\normm{\cdot}}$ implies weak convergence in $\L^2(\Omega)$, that $|\lambda_N-\lambda|\to0$ and that 
	$$\left|Q_U(\Psi,\Phi)-Q_{U^N}(\Psi,\Phi)\right|\leq K |\scalar{\psi}{(A_U-A_{U^N})\varphi}|\;,$$
	whose right hand side tends to zero by Assumption~\ref{ass:1}.
\end{proof}


\section{Finite element method approximation of the spectral problem}\label{sec:FEM}

The finite element model that we shall use has to be devised in order to satisfy Eq.~\eqref{eq:h1aprox}. Well known results in finite elements theory establish that piecewise linear continuous functions satisfy the convergence condition of Eq.\ \eqref{eq:conv.eign.aprox.h1} discussed in the previous section, cf.\ \cite{brenner08}. The construction of the algorithm will be split into two main parts, one considering the construction of the finite elements at the bulk, and other considering the finite elements at the boundary.
The bulk part will be a standard FEM using a uniform triangulation. Implementing the general boundary conditions, which is the main objective of this article, does not depend on the shape of the manifold $\Omega$ since the boundary will be a compact and connected manifold. Hence, we will consider the simple situation depicted in Fig.~\ref{Reticulo_ext}.  

\begin{figure}[h]
\centering
\includegraphics{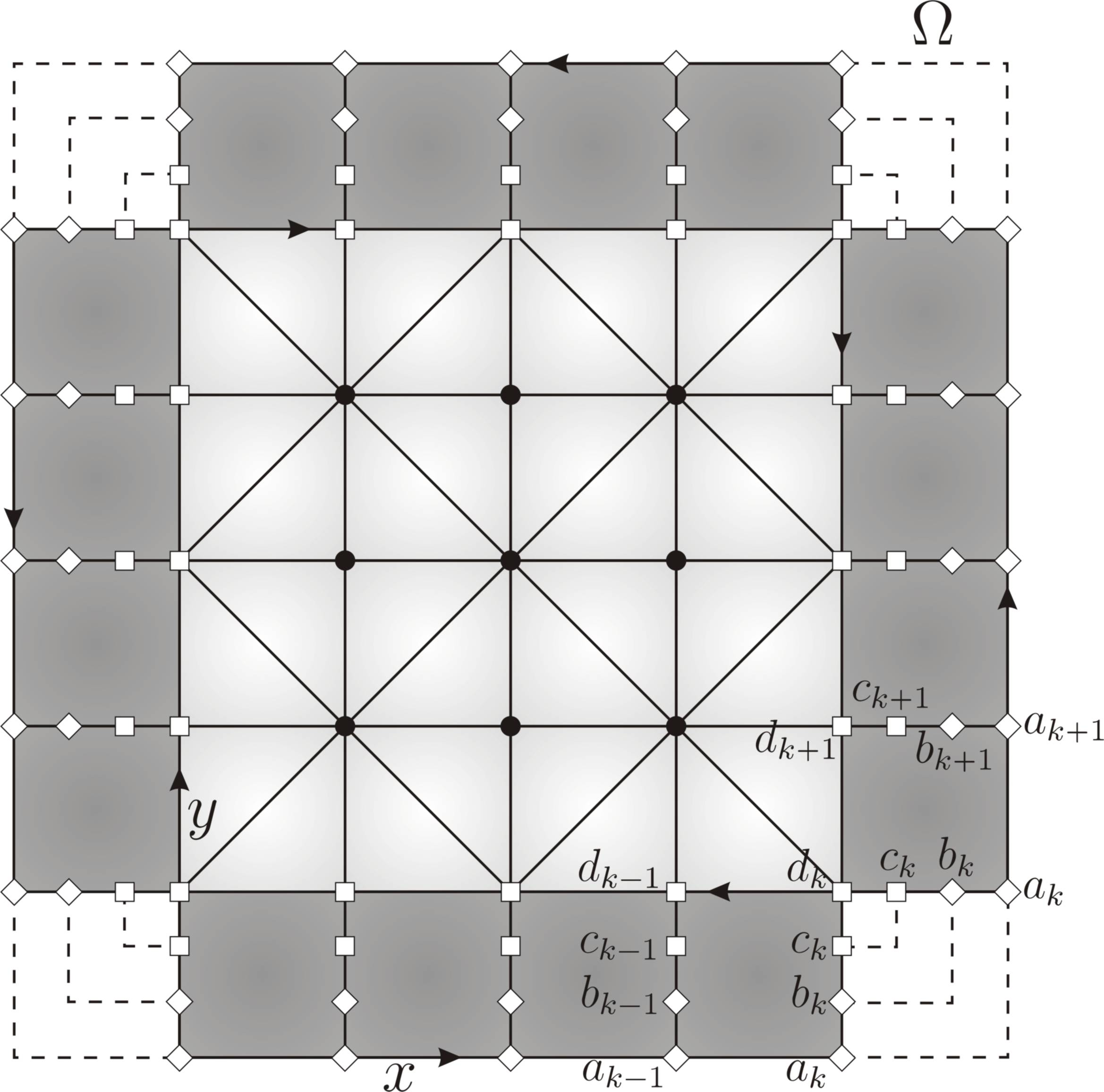}
\caption{\small Triangularisation of the problem for $n=m=5$. White region corresponds to the triangularised part at the bulk and the dark squares to the boundary. Regularisation at the corners is achieved by identification of the nodes at each corner. Bulk nodes are depicted in black. There are two types of nodes at the finite elements of the boundary. The interior nodes are problem nodes and are denoted with squares. The exterior nodes, called \emph{boundary nodes} are depicted with rhombi.}
\label{Reticulo_ext}
\end{figure}

As the region $\Omega$ we consider the square $[0,1]\times[0,1]$. We denote by $n$ and $m$ the number of node points in which the grid is divided, where $n$ denotes the number of nodes on the horizontal axis and $m$ the number of nodes on the vertical one. At the bulk we apply a standard FEM where the base-functions have value one at one node and zero at the rest. At the boundary, the finite elements will be squares and the base-functions will depend on the boundary conditions of the problem under consideration. The reason behind the choice of square elements at the boundary will become clear when dealing with the construction of the finite element models at the boundary. The regularisation procedure that we consider at the corners of the manifold $\Omega$ is the identification of the sides of the squares at the corners as shown in Fig.~\ref{Reticulo_ext}. Implicitly we are considering that the boundary is connected and that there are no preferred points.

We will consider for simplicity that the Riemannian metric is flat. The construction of the FEM at the bulk is in essence unaffected by this choice. Only the computation of the different scalar products, cf. Eqs.\ \eqref{eq:matricesM_B_F}, would be affected by the introduction of appropriate weights. On the other hand, the implementation of the boundary conditions should be modified slightly in order to consider non-flat metrics. In Section \ref{FEM:Boundary} we will point out which particular equations should be modified in order to implement the different boundary conditions in a non-flat and non-uniform case.

Let us reformulate the problem of Def.~\ref{DefQU} as it is usually done in FEM. First, we write the solution function $\Phi$ and the function $\Psi$ in terms of base-functions:
\begin{equation}
\Phi=\sum_{i=1}^{N}u_i\phi_i,\quad\qquad\Psi=\sum_{i=1}^{N}v_i\phi_i,
\end{equation}
where $N$ is the total number of base-functions. Hence, \eqref{approximatespectralproblem} becomes:
\begin{equation}
\sum_{i,j=1}^N\left[u_i\left(M_{ij}-F_{ij}\right)v_j\right]=\lambda\sum_{i,j=1}^N u_i B_{ij}v_j,\qquad\forall v_j,
\end{equation}
where the matrices $M$, $B$ and $F$ are defined as:
\begin{subequations}\label{eq:matricesM_B_F}
\begin{equation}\label{matrices_M_B}
M_{ij}=\int_\Omega\left(\phi_{ix}\phi_{jx}+\phi_{iy}\phi_{jy}\right)\mathrm{d}\mu_{\eta},\qquad B_{ij}=\int_{\Omega}\phi_i\phi_j\mathrm{d}\mu_{\eta},
\end{equation}
and
\begin{equation}\label{matrix_F}
F_{ij}=\int_{\partial\Omega}\left(\phi_{ix}n_1+\phi_{iy}n_2\right)\phi_j\mathrm{d}\mu_{\partial\eta},
\end{equation}
\end{subequations}
where the subindexes $x$ and $y$ denote the partial derivatives with respect to $x$ and $y$ respectively, and $\textbf{n}=(n_1,n_2)$ is the normal vector to the boundary. 
We are going to treat essential and natural boundary conditions in a unified way. Therefore, the matrix $F$ represents the boundary term $\scalar{\varphi}{\dot{\varphi}}_{\pO}$ coming from the integration by parts formula. We explain the reasons behind this choice in Section~\ref{FEM:Boundary}. The problem that we solve numerically becomes then:
\begin{equation}\label{eq:geneigenvprob}
(M-F)\textbf{u}=\lambda B\textbf{u},
\end{equation}
with $\textbf{u}=(u_1,u_2,\ldots,u_N)$.


\subsection{Finite elements at the bulk}

The structure of the triangularisation in Fig. \ref{Reticulo_ext} makes it simple to compute the contributions of the finite elements at the bulk. As we are interested in showing the performance of the algorithm based on the boundary elements, we will consider just linear base-functions $\phi_i$, $i=1,\ldots,N_{Bulk}$ at the bulk, leaving higher order finite elements for future implementations. The base-function corresponding to the $i$-th node has value $1$ in that node and zero at the rest is:
\begin{equation}\label{def_base_func}
\phi_i(\mbox{node}_j)=\left\{\begin{matrix}
1& i=j\\
0&i\neq 0
\end{matrix}\right.\,.
\end{equation}

\begin{figure}[h]
    \centering%
    \begin{subfigure}[t]{0.3\textwidth}
        \centering%
         \raisebox{8pt}[0pt][0pt]{\includegraphics[height = 2.5cm]{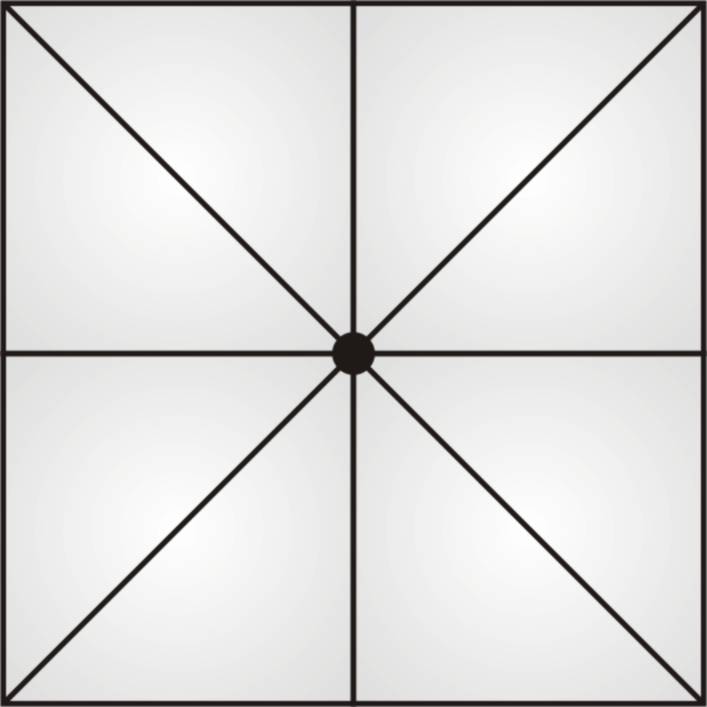}}%
        \caption{Star-like}\label{fig:starnode}%
    \end{subfigure}%
    \begin{subfigure}[t]{0.3\textwidth}
        \centering%
        \raisebox{8pt}[0pt][0pt]{\includegraphics[height = 2.5cm]{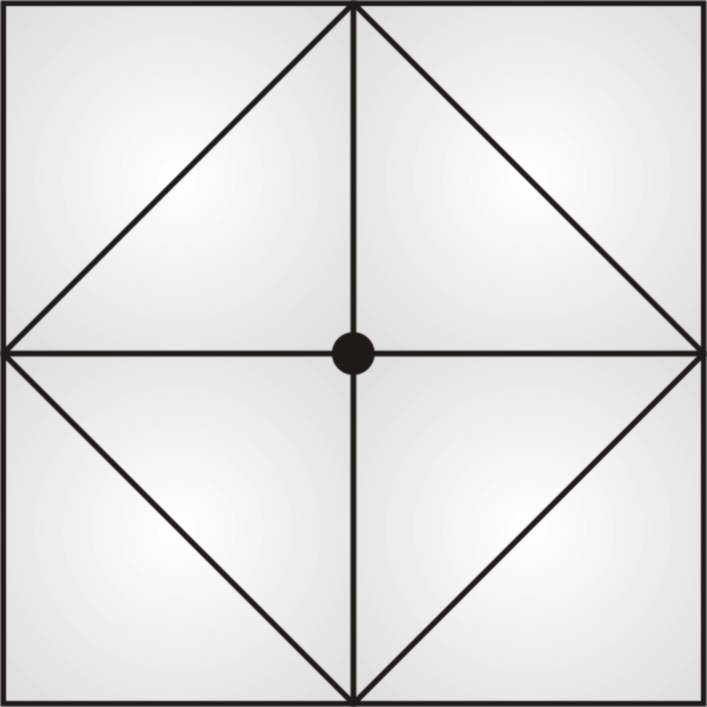}}%
        \caption{Diamond-like}\label{fig:diamondnode}%
    \end{subfigure}%
    \begin{subfigure}[t]{0.35\textwidth}
        \centering%
        {\includegraphics[height=3cm]{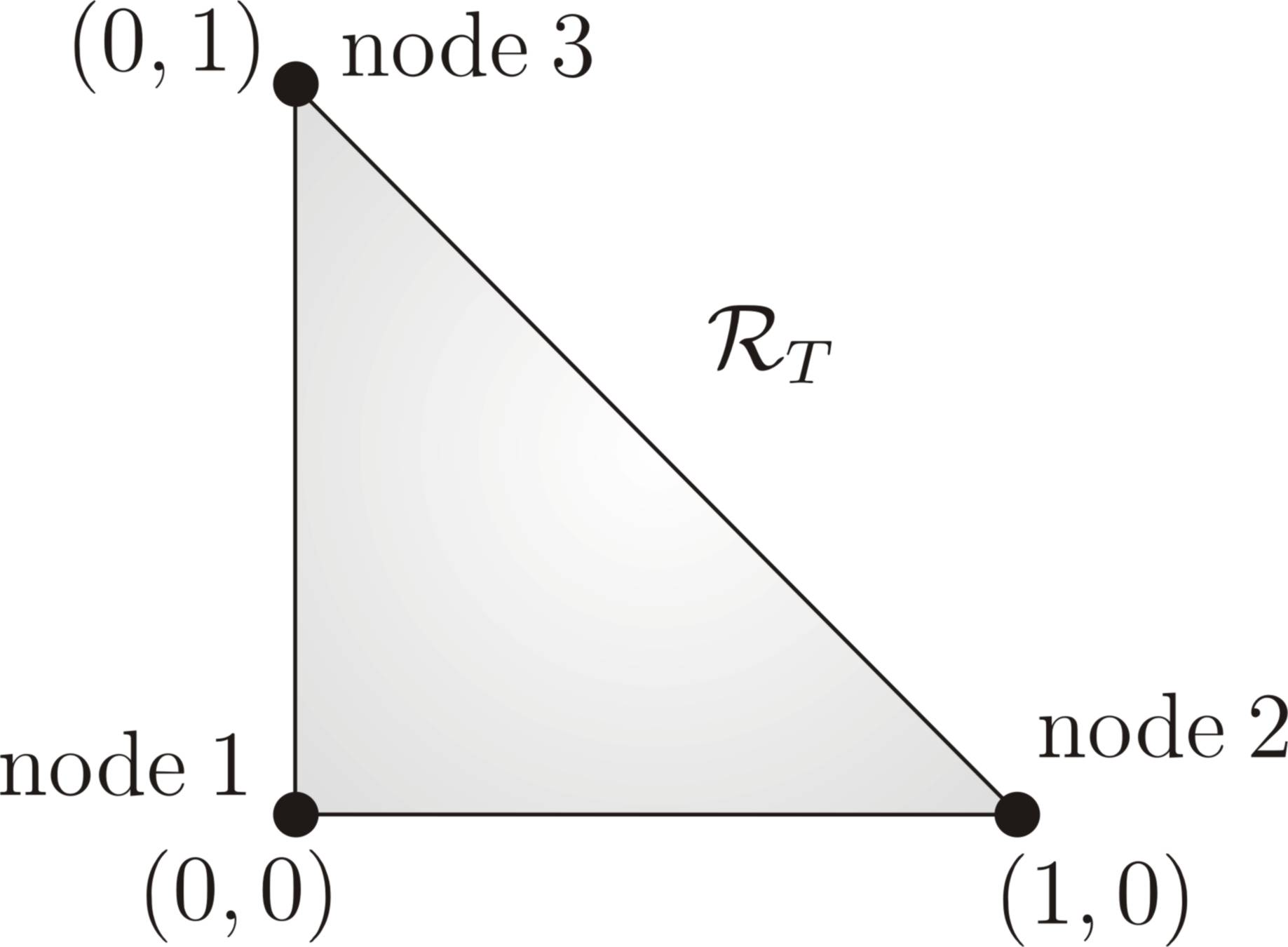}}%
        \caption{Reference triangle $\mathcal{R}_T$}\label{Ref_triang}%
    \end{subfigure}%
    \caption{\small The two different kind of nodes of the bulk and the reference triangle.}\label{two_nodes}
\end{figure}

In Fig. \ref{Reticulo_ext}, it can be seen that there are two kind of nodes at the bulk, the star-like and the diamond-like. Those are depicted respectively in Fig.~\ref{fig:starnode} and Fig.~\ref{fig:diamondnode}. 
Having this into account, it is easy to compute the matrix elements of $M$ and $B$ corresponding to the bulk.
We have chosen the reference triangle with coordinates $(\eta,\xi)$ showed in Fig. \ref{Ref_triang}.
The affine transformation that transforms the reference triangle $\mathcal{R}_T$ in any triangle with vertices $(x_1,y_1)$, $(x_2,y_2)$ and $(x_3,y_3)$ is the following:
\begin{equation}\label{trans_reference_tri}
\begin{pmatrix}
x\\
y
\end{pmatrix}=\begin{pmatrix}
x_1\\
y_1
\end{pmatrix}+\begin{pmatrix}
x_2-x_1&x_3-x_1\\
y_2-y_1&y_3-y_1
\end{pmatrix}\begin{pmatrix}
\eta\\
\xi
\end{pmatrix}.
\end{equation}
Hence, the elements of the first derivatives of the coordinates in the reference triangle with respect to $x$ and $y$ are given by: 
\begin{equation}
\begin{pmatrix}
\eta_x& \eta_y\\
\xi_x&\xi_y
\end{pmatrix}=\frac{1}{J_T}\begin{pmatrix}
y_3-y_1&x_1-x_3\\
y_1-y_2&x_2-x_1
\end{pmatrix},
\end{equation}   
where the Jacobian is given by  $J_T=|(x_2-x_1)(y_3-y_1)-(x_1-x_3)(y_1-y_2)|$. 
Let us denote by $(\alpha_i, \beta_i, \gamma_i)$ the coefficients of the base-function $\phi_i$ restricted to a given triangle, that is
\begin{equation}\label{eq:basefunctionbulk}
	\phi_i|_T = \alpha_i + \beta_i x + \gamma_i y\;.
\end{equation}
Then, the results of the integrals \eqref{matrices_M_B} of the matrix elements of $M$ and $B$, corresponding to any triangle in our triangularisation, cf.\ Fig. \ref{Reticulo_ext}, are:
\begin{subequations}\label{Bulk_matrices}
\begin{equation}\label{M_ij_one_triangle}
\hspace{-0.35cm}\left.M_{ij}\right|_{T}\hspace{-0.03cm}=\hspace{-0.03cm}\frac{J_T}{2}\big[(\beta_i\eta_x+\gamma_i\xi_x)(\beta_j\eta_x+\gamma_j\xi_x)\\
\left.+(\beta_i\eta_y+\gamma_i\xi_y)(\beta_j\eta_y+\gamma_j\xi_y)\big]\right|_T,
\end{equation}
\begin{equation}\label{B_ij_one_triangle}
\hspace{-0.04cm}\left.B_{ij}\right|_{T}\hspace{-0.03cm}=\hspace{-0.03cm}\frac{J_T}{2}\left[\alpha_i\alpha_j\hspace{-0.05cm}+\hspace{-0.05cm}\frac{\alpha_i\beta_j\hspace{-0.05cm}+\hspace{-0.05cm}\beta_i\alpha_j}{3}\hspace{-0.05cm}+\hspace{-0.05cm}\frac{\alpha_i\gamma_j\hspace{-0.05cm}+\hspace{-0.05cm}\gamma_i\alpha_j}{3}\right.
\hspace{-0.1cm}\left.\left.+\frac{\beta_i\gamma_j\hspace{-0.05cm}+\hspace{-0.05cm}\gamma_i\beta_j}{12}\hspace{-0.05cm}+\hspace{-0.05cm}\frac{\beta_i\beta_j}{6}\hspace{-0.05cm}+\hspace{-0.05cm}\frac{\gamma_i\gamma_j}{6}\right]\right|_T\!\!.
\end{equation}
\end{subequations}
The matrix elements of $M$ and $B$ are computed by adding the latter results over all the triangles in the triangularisation:
\begin{equation}\label{Bulk_matrices_sum}
M_{ij}=\sum_{T}\left.M_{ij}\right|_{T}\,,\qquad B_{ij}=\sum_{T}\left.B_{ij}\right|_{T}.
\end{equation}
Because of the definition of the base-functions at the bulk, if $i\neq j$ the former additions restrict to the two triangles that share the edge between the nodes $i$ and $j$. 
If $i=j$, it is obvious that the support of the integrands of $M_{ii}$ and $B_{ii}$ is the same as the support of the base-function located at $i$, hence we will have two kinds of matrix elements depending if $i$ is a star-like or a diamond-like node.  In the former case the sum will restrict to eight triangles and in the latter to four, see Fig.~\ref{two_nodes}.


\subsection{Finite elements at the boundary}\label{FEM:Boundary}

As stated at the beginning of this section, piecewise linear continuous functions provide the appropriate convergence with respect to the Sobolev norms, cf.\ \cite{brenner08}. We will consider a finite element model that guarantees that the restriction to the boundary of the functions, as well as their normal derivatives, are piecewise linear continuous functions. This condition implies that the finite element model on the region at the boundary (dark region in Fig.\ \ref{Reticulo_ext}) cannot consist on linear polynomials. This is so because the normal derivative of a piecewise linear continuous function is a piecewise constant function. Notice that going to higher order polynomials does not spoil the convergence in the Sobolev norm of the FEM as long as continuity is preserved between finite elements.

A convenient way of imposing the boundary conditions is as follows: Instead of using the operators $P_{U^N}$ and $A_{U^N}$ to construct  the approximate problem directly, we will construct functions at the boundary that satisfy the boundary condition
\begin{equation}\label{eq:approxbc}
	\varphi - i \dot{\varphi}= U^N (\varphi + i \dot{\varphi})\;,
\end{equation}
where $P_{U^N}$ is the orthogonal projector onto the proper space of $U^N$ associated to the eigenvalue $\{-1\}$, $A_{U^N}$ is the partial Cayley transform of Def.~\ref{partialCayley}, $\varphi = \gamma(\Phi)$ is the restriction of $\Phi$ to the boundary, and $\dot{\varphi} = \gamma(\frac{\d\Phi}{\d\mathbf{n}})$ is the restriction of the normal derivative. So far there has not been any assumption on the flatness of the Riemannian metric. The following construction can be applied also to the non-flat case. One has then to take into account that the normal derivative would me be measured with respect to the Riemannian metric and that the measures along the boundary need to be modified accordingly.

Notice that Eq.~\eqref{eq:approxbc} splits into two equations, one on the proper subspace of $U^N$ associated to $\{-1\}$ and one on its orthogonal complement:
\begin{subequations}
\begin{align}
	P_{U^N}\varphi - i P_{U^N}\dot{\varphi} = - (P_{U^N}\varphi + i P_{U^N}\dot{\varphi} ) \quad &\!\!\!\Leftrightarrow\!\!\! \quad P_{U^N}\varphi = 0\,,\\
	P_{U^N}^\bot\varphi - i P_{U^N}^\bot\dot{\varphi} = U^N\bigr|_{\mathrm{ran}P_{U^N}^\bot} (P_{U^N}^\bot\varphi + i P_{U^N}^\bot\dot{\varphi} ) \quad &\!\!\!\Leftrightarrow\!\!\! \quad P_{U^N}^\bot\dot{\varphi} = iP_{U}^\bot\frac{ (U-\mathbb{I})}{ U+\mathbb{I}}P_{U^N}^\bot\varphi\,.
\end{align}
\end{subequations}
That is, the problem defined by 
	$$\scalar{\d\Phi}{\d\Phi} - \scalar{\varphi}{\dot{\varphi}} = \lambda \scalar{\Phi}{\Phi}\,,$$ 
such that $\Phi$ verifies the boundary condition in Eq.~\eqref{eq:approxbc}, is equivalent to the approximate problem of Eq.~\eqref{approximatespectralproblem}. Notice that there is a one-to-one correspondence between unitary operators $U^N$ and the operators $P_{U^N}$ and $A_{U^N}$. 

The finite element model that we will consider at the boundary shall be constructed in such a way that the restrictions to the boundary of the base-functions are piecewise linear continuous functions. That is, for the interval $I_j = [x_j,x_j+h]$ each boundary function will have the following restrictions to the boundary: 
\begin{subequations}\label{eq:boundarylinear}
    \begin{align}
        L^j(x) &= \frac{1}{h}(a_{j+1} - a_j) x - \frac{1}{h} (a_{j+1} - a_j) x_j + a_j\,,\\
        \dot{L}^j(x) &= \frac{1}{h}(n_{j+1} - n_j) x - \frac{1}{h} (n_{j+1} - n_j) x_j + n_j\,,
    \end{align}
\end{subequations}
where, as stated above, the interval length $h$ is assumed to be the same at all the intervals.
The former function $L^j(x)$ is the restriction of the base-function to the interval $I_j$, where $a_j$ and $a_{j+1}$ are the values of the base-function at the endpoints of the interval $I_j$\,. The latter are determined by $x_j$ and $x_{j+1}= x_j+h$ respectively. Similarly, $\dot{L}^j(x)$ is the linear function at the interval $I_j$ corresponding to the restriction to the boundary of the normal derivative of the base-function. In this case, $n_j$ and $n_{j+1}$ are the values of the normal derivatives at the nodes $x_j$ and $x_{j+1}$ respectively.

In order to represent the unitary operator $U^N$, we will choose an orthonormal basis given in terms of the Legendre polynomials of order $0$ and $1$ defined over each interval. Assume that there are $N_S$ different intervals at the boundary, i.e. $N_S$ different squares, and consider the basis given by
\begin{equation}\label{eq:legendrep}
	P_0^j(x)=\frac{1}{\sqrt{h}}\,,\;\; P^j_1(x)=\frac{2\sqrt{3}}{\sqrt{h^3}}(x-x_j-\frac{h}{2})\,,\;\; x \in [x_j,x_j + h]\,,\;\; 1\leq j\leq N_S\,.
\end{equation}
These functions are extended by zero to the rest of the boundary, i.e. $P_0^j(x) = P^j_1(x) = 0$ for $x\notin I_j$.
In this basis, the boundary condition of Eq.~\eqref{eq:approxbc} will take the form
\begin{equation}\label{eq:legendrebc}
    \xi_k - i \zeta_k = U_{kl} (\xi_l + i \zeta_l)\,,\qquad k,l=1,\ldots,2N_S\,,
\end{equation}
where $ \xi_k$ and $\zeta_k$ are the coefficients of the Legendre expansion and are related to the linear functions of Eqs.~\eqref{eq:boundarylinear} by
\begin{subequations}\label{eq:legendrecoefficients}
    \begin{alignat}{2}
        \xi_j = \scalar{P_0^j}{L^j}_{I_j} &= \int_{x_j}^{x_j+h} P_0^j(x) L^j(x)\d x = \frac{\sqrt{h}}{2}(a_{j+1} + a_j)\,,\qquad&1\leq j\leq N_S\,,\\
        \xi_{j+N_S} = \scalar{P_1^j}{L^j}_{I_j} &= \int_{x_j}^{x_j+h} P_1^j(x) L^j(x)\d x = \frac{\sqrt{h}}{2\sqrt{3}}(a_{j+1} - a_j)\,,\qquad&1 \leq j\leq N_S\,,\\
        \zeta_j = \scalar{P_0^j}{\dot{L}^j}_{I_j} &= \int_{x_j}^{x_j+h} P_0^j(x) \dot{L}^j(x)\d x = \frac{\sqrt{h}}{2}(n_{j+1} + n_j)\,,\qquad& 1\leq j\leq N_S\,,\\
        \zeta_{j+N_S} = \scalar{P_1^j}{\dot{L}^j}_{I_j} &= \int_{x_j}^{x_j+h} P_1^j(x) \dot{L}^j(x)\d x = \frac{\sqrt{h}}{2\sqrt{3}}(n_{j+1} - n_j)\,,\qquad& 1\leq j\leq N_S\,.
    \end{alignat}
\end{subequations}
If one is considering a Riemannian metric that is not uniform at the boundary, one needs to include the Riemannian measures at each interval and use a basis of orthogonal polinomials, see Eq. \eqref{eq:legendrep}, adapted to these different weights.
The unitary matrix $U_{kl}$ is defined by 
\begin{equation}\label{eq:unitarylegendre}
    U_{kl} =
        \begin{cases}
            \scalar{P_0^k}{U^N P_0^l}\,, & 1 \leq k , l \leq N_S\,, \\
            \scalar{P_1^{k-N_S}}{U^N P_0^l}\,, & N_S + 1 \leq k \leq 2N_S\,,\quad 1 \leq l \leq N_S\,,\\
            \scalar{P_0^k}{U^N P_1^{l-N_S}}\,, & 1 \leq k \leq N_S\,,\!\!\qquad N_S +1 \leq  l \leq 2N_S\,, \\
            \scalar{P_1^{k-N_S}}{U^N P_1^{l-N_S}}\,, & N_S + 1 \leq k,  l \leq 2N_S\,. \\
        \end{cases}\;
\end{equation}

As stated before, we need a finite element model where the restrictions to the boundary are linear functions and whose restrictions to the boundary of their normal derivatives are linear too. Moreover, the restriction of the functions at the border with the bulk region (see Fig.\ \ref{Reticulo_ext}) have also to be linear. This latter condition guarantees a continuous matching with the base-functions at the bulk. To achieve this, we have chosen the following model in a square with eight nodes. The reference square $\mathcal{R}_S$ is depicted in Fig.~\ref{Ref_square}. 
We will have on each square the following base-function:
\begin{equation}\label{eq:refpolynomial}
   \left.\phi_i\right|_S = p_{1i}y^3x+p_{2i}y^3+p_{3i}y^2x+p_{4i}y^2 +p_{5i}yx+p_{6i}y+p_{7i}x+p_{8i}\,,
\end{equation}
where $x$ is the coordinate parallel to the boundary and $y$ is the coordinate normal to it. The coordinates with $ y = 0$ will correspond to the boundary of the region while $ y = 1$ represents the region of the square that touches the bulk. Notice that the polynomial is at most linear in $x$ and that the derivative with respect to $y$ is also linear in $x$.
We shall consider that the pairs of interior nodes (denoted by $\Box$) are problem nodes and that the pairs of exterior nodes (denoted by $\Diamond$), called \emph{boundary nodes}, are going to be determined by the boundary conditions and the values at the problem nodes. Notice that the choice of a polynomial of up to second order in the variable $y$ would have given an unequal number of problem nodes and boundary nodes. 

With this model, the value of the normal derivative at the node $x_j$, in a square of side $h$, is given by the expression:
\begin{equation}\label{eq:normalderivative}
    n_j = \frac{1}{h}\left( -\frac{11}{2}a_j + 9b_j - \frac{9}{2}c_j + d_j \right)\,.
\end{equation}
Notice that Eq.~\ref{eq:normalderivative} would need to be modified in the case of non-flat metrics in order to obtain the proper correlation between the value of the normal derivative at the boundary and the value of the base functions at the nodes.

\begin{figure}[h]
    \centering
    \includegraphics[height = 3.5cm]{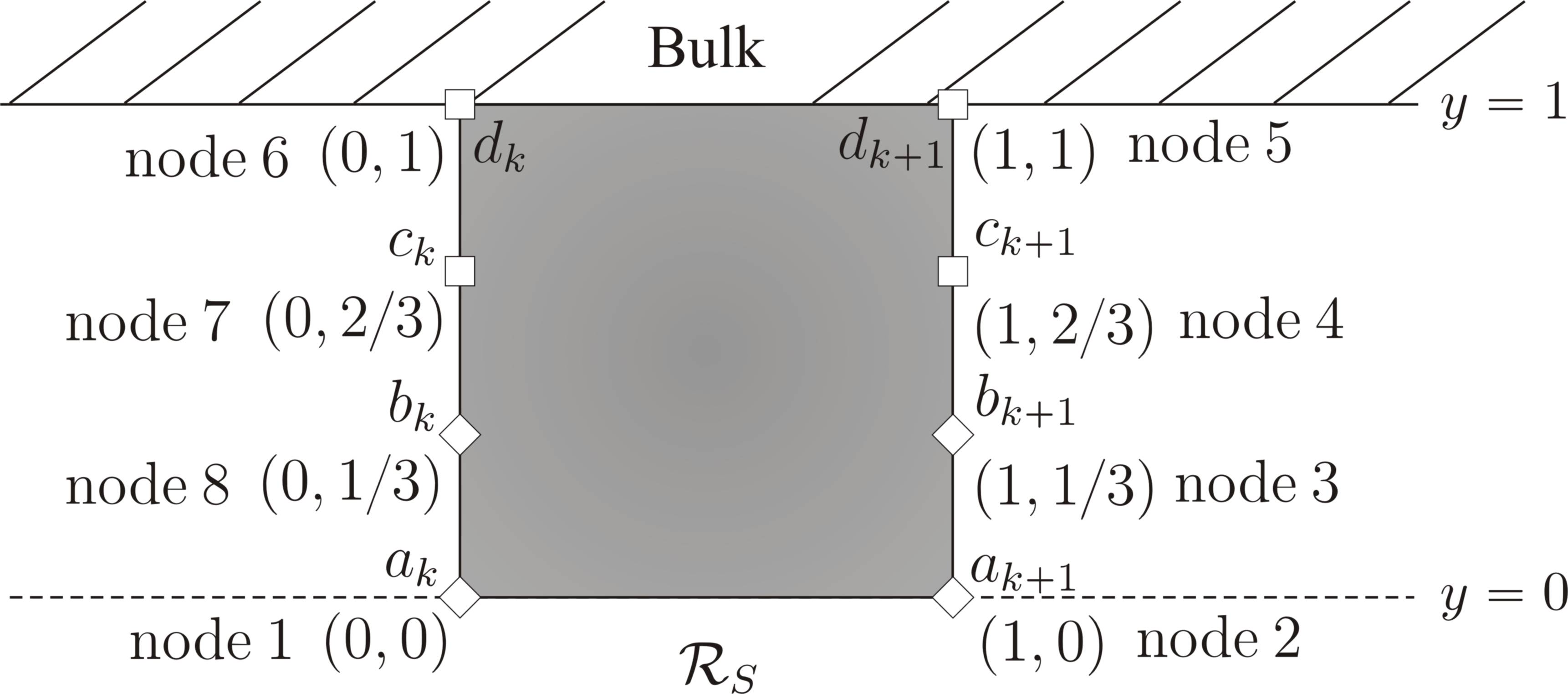}
    \caption{\small Reference square $\mathcal{R}_S$. Parameters $a$, $b$, $c$, $d$ are the value of the base-function at those nodes.  Parameters $a_k$ and $a_{k+1}$ represent the values of the base-function at the boundary, and parameters $d_k$ and $d_{k+1}$ are the values at the region touching the bulk. The $\Box$ denote the problem nodes and the $\Diamond$ denote the boundary nodes.}
    \label{Ref_square}
\end{figure}

Now we will use the boundary conditions, Eq.~\eqref{eq:legendrebc}, to determine the value at the boundary nodes in terms of the problem nodes. Hence, we will substitute back Eq.~\eqref{eq:normalderivative} in Eqs.~\eqref{eq:legendrecoefficients} and Eq.~\eqref{eq:legendrebc}. Denoting by $s$ a column vector with the values of $a_j$ and $b_j$, and by $w$ a column vector with the values of $c_i$ and $d_i$\,
\begin{equation}\label{eq:sw}
    s^\top = [a_1,\dots,a_{N_S}, b_1,\dots, b_{N_S}]\,, \qquad w^\top = [c_1,\dots, c_{N_S}, d_1,\dots, d_{N_S}]\,,
\end{equation}
and having into account that at the boundary we have identified node $x_0$ with node $x_{N_S}$, we get the following equations in matrix form:
\begin{equation}\label{eq:coefficientindependent}
	\begin{bmatrix}
		\mathcal{F}_0 \\ \mathcal{F}_1   
	\end{bmatrix}
	s =
	\begin{bmatrix}
		\mathcal{C}_0 \\ \mathcal{C}_1  
	\end{bmatrix}
	w,
\end{equation}
with
\begin{subequations}\label{eq:linearsystemcomplete}
\begin{align}
	\mathcal{F}_0 &=
	         \left( \frac{\sqrt{h}}{2} + i\frac{11}{4\sqrt{h}} \right)(\mathcal{N}+\mathcal{I})\mathcal{L} - i\frac{9}{2\sqrt{h}}(\mathcal{N}+\mathcal{I})\mathcal{R}\nonumber\\
                &\hspace{0.2cm}\phantom{aaaa}- U_{[00]}\left[\left( \frac{\sqrt{h}}{2} - i\frac{11}{4\sqrt{h}}\right)(\mathcal{N}+\mathcal{I})\mathcal{L} - i\frac{9}{2\sqrt{h}}(\mathcal{N}+\mathcal{I})\mathcal{R} \right]\nonumber\\
                &\hspace{0.2cm}\phantom{aaaa}-U_{[01]}\left[\left( \frac{\sqrt{h}}{2\sqrt{3}} - i\frac{11}{4\sqrt{3}\sqrt{h}}\right)(\mathcal{N}-\mathcal{I})\mathcal{L} + i\frac{9}{2\sqrt{3}\sqrt{h}}(\mathcal{N}-\mathcal{I})\mathcal{R} \right],
	\\
	\mathcal{F}_1 &=
	         \left( \frac{\sqrt{h}}{2\sqrt{3}} + i\frac{11}{4\sqrt{3}\sqrt{h}} \right)(\mathcal{N}-\mathcal{I})\mathcal{L} - i\frac{9}{2\sqrt{3}\sqrt{h}}(\mathcal{N}-\mathcal{I})\mathcal{R}\nonumber\\
        	  &\hspace{0.1cm}\phantom{aaaa}- U_{[10]}\left[\left( \frac{\sqrt{h}}{2} - i\frac{11}{4\sqrt{h}}\right)(\mathcal{N}+\mathcal{I})\mathcal{L} - i\frac{9}{2\sqrt{h}}(\mathcal{N}+\mathcal{I})\mathcal{R} \right]\notag\\
                &\hspace{0.1cm}\phantom{aaaa}-U_{[11]}\left[\left( \frac{\sqrt{h}}{2\sqrt{3}} - i\frac{11}{4\sqrt{3}\sqrt{h}}\right)(\mathcal{N}-\mathcal{I})\mathcal{L} + i\frac{9}{2\sqrt{3}\sqrt{h}}(\mathcal{N}-\mathcal{I})\mathcal{R} \right],
        \\
        \mathcal{C}_0 &=
                -i\frac{9}{4\sqrt{h}}(\mathcal{N}+\mathcal{I})\mathcal{L} + i \frac{1}{2\sqrt{h}}(\mathcal{N}+\mathcal{I})\mathcal{R}\nonumber\\
                &\hspace{1.3cm}\phantom{aaaa} + U_{[00]}\left[-i\frac{9}{4\sqrt{h}}(\mathcal{N}+\mathcal{I})\mathcal{L} + i\frac{1}{2\sqrt{h}}(\mathcal{N}+\mathcal{I})\mathcal{R}\right]\nonumber\\
                &\hspace{1.3cm}\phantom{aaaa} + U_{[01]}\left[-i\frac{9}{4\sqrt{3}\sqrt{h}}(\mathcal{N}-\mathcal{I})\mathcal{L} + i\frac{1}{2\sqrt{3}\sqrt{h}}(\mathcal{N}-\mathcal{I})\mathcal{R}\right]\,,
        \\
        \mathcal{C}_1 &=
                -i\frac{9}{4\sqrt{3}\sqrt{h}}(\mathcal{N}-\mathcal{I})\mathcal{L} + i \frac{1}{2\sqrt{3}\sqrt{h}}(\mathcal{N}-\mathcal{I})\mathcal{R}\notag\\
            &\hspace{1.35cm}\phantom{aaaa} + U_{[10]}\left[-i\frac{9}{4\sqrt{h}}(\mathcal{N}+\mathcal{I})\mathcal{L} + i\frac{1}{2\sqrt{h}}(\mathcal{N}+\mathcal{I})\mathcal{R}\right]\notag\\
            &\hspace{1.35cm}\phantom{aaaa} + U_{[11]}\left[-i\frac{9}{4\sqrt{3}\sqrt{h}}(\mathcal{N}-\mathcal{I})\mathcal{L} + i\frac{1}{2\sqrt{3}\sqrt{h}}(\mathcal{N}-\mathcal{I})\mathcal{R}\right]\,,\\\notag
\end{align}
\end{subequations}

\noindent where
\begin{equation}
	U = 
	\begin{pmatrix}
		U_{[00]} & U_{[01]} \\ U_{[10]} & U_{[11]}
	\end{pmatrix}
\end{equation}
is the block-wise decomposition of $U$, with blocks of size $N_S\times N_S$ that corresponds to the block-wise structure of Eq.~\eqref{eq:unitarylegendre}. The matrix 
 $\mathcal{I}$ is the $N_S\times N_S$ identity matrix, $\mathcal{N}$ is the $N_S\times N_S$ row cyclic-shift matrix defined as follows:
$$
    \mathcal{N} = \begin{pmatrix}
        0 & 1 & 0 & \cdots & 0 \\
        \vdots & \diagdots[-39]{1.2em}{.11em} & \hspace{0.04cm}\diagdots[-39]{1.2em}{.11em} & & \vdots\\
        \vdots & & \diagdots[-39]{1.2em}{.11em} & \diagdots[-39]{1.2em}{.11em} & 0 \\
        0 & & & \hspace{-0.1cm}\diagdots[-39]{1.2em}{.11em} & 1 \\
        1 & 0 & \cdots& \cdots & 0
    \end{pmatrix}\;.
$$
The matrix $\mathcal{L}$ is the $N_S\times 2N_S$ matrix whose first $N_S$ columns are the identity $N_S\times N_S$ matrix and the rest is zero $\mathcal{L}=[\mathcal{I}\;,\textbf{0}]$. Similarly, $\mathcal{R} = [\textbf{0}, \mathcal{I}]$ is the $N_S\times 2N_S$ matrix whose last $N_S$ columns are the identity matrix. 

Given a vector $w$, i.e.\ the vector with the values of the nodes $c_j$ and $d_j$ of Eq.~\eqref{eq:sw}, the linear system defined in Eqs.~\eqref{eq:linearsystemcomplete} determines a linear system of equations with unknown $s$\,, i.e.\ the vector with the values of the nodes $a_j$ and $b_j$\,. Therefore, we need to choose a set of vectors $\{w^k\}_{1\leq k\leq 2N_S}$ furnishing a set of linearly independent functions. The most natural choice for this set is the canonical  basis of $\mathbb{R}^{2N_S}$\,. This choice leads to the set of linear systems
\begin{equation}\label{eq:linearsystemcompact}
	\mathcal{F}s = \mathcal{C}\,,
\end{equation} 
with $\mathcal{F}, \mathcal{C}\in \mathbb{C}^{2N_S\times 2N_S}$\, the coefficient and independent matrices of the linear system in Eq.~\eqref{eq:coefficientindependent}. We use the same symbol $s$ to denote the solutions of this family of linear systems. This family depends on the unitary operator $U$ and is overdetermined. In order to solve it, we will use the singular value decomposition (SVD for short) of the coefficient matrix $\mathcal{F}$\,, cf.\ \cite{demmel97}\,.

The fact that the matrix $\mathcal{F}$ is not of full rank tells us that there is a freedom in the choice of $2N_S -\mathrm{rank(\mathcal{F})}$ parameters. For every choice of these parameters, one gets an alternative set of solutions of the linear system \eqref{eq:linearsystemcompact}. Instead of setting all these parameters to zero, we will additionally solve a least squares problem to fix these parameters in such a way that the solution is the one that deviates the less from the Neumann case. The Neumann case is taken as a reference set of boundary functions with value $a_j = 1$ at one node of the boundary and $a_k = 0$, $k \neq j$ at the remaining nodes of the boundary. The coefficients $\{b_j\}$ are chosen such that the normal derivatives, given by Eq.~\eqref{eq:normalderivative}, vanish. 
For the choices of the coefficients $\{c_j\}$ and $\{d_j\}$ as the canonical basis of $\mathbb{R}^{2N_S}$, the reference matrix $X_{\mathrm{Neumann}}$ is as follows
\begin{equation}\label{Neumann_matrices_alg}
    X_{\mathrm{Neumann}} = \left[ X_c\;\; X_d \right]\;, \quad X_c =
            \begin{bmatrix}
                \mathcal{I} \\ \frac{10}{9}\mathcal{I}
            \end{bmatrix}
            \;, \quad X_d =
            \begin{bmatrix}
                \mathcal{I} \\ \frac{1}{2}\mathcal{I}
            \end{bmatrix}\;.
\end{equation}
The reason for looking for the least squares approximation that better fits $X_{\mathrm{Neumann}}$ is two-fold:

First, it minimises the number of boundary elements where the boundary functions are different from zero. The performance of the algorithm is greatly increased in this way since it reduces polynomially the number of operations needed to compute the matrices of Eq.~\eqref{eq:geneigenvprob}. Recall that, in general, the boundary functions are non-local and spread along the boundary. Thus, the computations of the matrices $M$, $F$ and $B$ involve sums over all the squares of the boundary. Although, in principle, there are situations with non-local boundary conditions that necessarily spread along the full boundary, in practice the most usual situations require non-vanishing of the boundary functions in a very limited number of boundary elements. For instance, periodic and quasi-periodic boundary conditions can be implemented with boundary functions that are non-vanishing only at a pair of opposed sites of the grid. 

Second, the least squares problem solution tries to get the non-vanishing boundary values as close to the value 1 as possible, which improves the behaviour of the implementation of natural boundary conditions. In addition to this, the computation of the least squares problem solution can be achieved without almost no additional computational cost since the factorisation provided by the singular value decomposition can be used directly to compute the least squares solution.

Let $\mathcal{F}= \mathcal{W}\mathcal{S}\mathcal{V}^\dagger$ be the SVD of the matrix $\mathcal{F}$\,. The relation between the complete and reduced SVD factorisations is given in block-wise matrix notation as:
\begin{equation}
	\mathcal{W}= :\left[\mathcal{W}_{\mathrm{red}}, \mathcal{W}_{\mathrm{null}}\right]\,,\quad
	\mathcal{S} =: 
		\begin{bmatrix}
			\mathcal{S}_{\mathrm{red}} & 0 \\ 0 & 0
		\end{bmatrix}\,,\quad
	\mathcal{V}=:\left[\mathcal{V}_{\mathrm{red}}, \mathcal{V}_{\mathrm{null}}\right]\;.
\end{equation}
Before solving the system one has to check that the linear system is compatible within numerical error, i.e. 
\begin{equation}\label{check_system}
\mathcal{W}_{\mathrm{null}}^\dagger \mathcal{C} = 0\;.
\end{equation}
Deviations from this serve as a measure of the goodness of the solution.
The inhomogeneous solution of the linear system \eqref{eq:linearsystemcompact} is given by $s_{ih} = \mathcal{V}_{\mathrm{red}}x$\,, where $x$ is the solution of the reduced linear system
\begin{equation}
	\mathcal{S}_{\mathrm{red}}x = \mathcal{W}_{\mathrm{red}}^\dagger\mathcal{C}\;.
\end{equation}
The general solution of the linear system is therefore given by
\begin{equation}\label{eq:sollinearsystem}
	s = \mathcal{V}_{\mathrm{red}}x + \mathcal{V}_{\mathrm{null}}\varepsilon\;,
\end{equation}
where $\varepsilon\in\mathbb{C}^{2N_S-\mathrm{rank(\mathcal{F})}}$ is the family of parameters that we will fix by means of the least squares solution\,. Among all the possible solutions, we want the one that minimises 
\begin{equation}\label{minimise_s}
\norm{s-X_{\mathrm{Neumann}}}^2\;,
\end{equation}
where the norm is the euclidean norm in $\mathbb{R}^{2N_S}$\,. The least squares solution is 
\begin{equation}\label{eq:lqs}
	\varepsilon = \mathcal{V}_{\mathrm{null}}^\dagger(X_{\mathrm{Neumann}}-\mathcal{V}_{\mathrm{red}}x ) 
			=  \mathcal{V}_{\mathrm{null}}^\dagger X_{\mathrm{Neumann}}\;.
\end{equation}
Substituting back in Eq.~\eqref{eq:sollinearsystem} we get finally
\begin{equation}\label{eq:boundarysolution}
	s = \mathcal{V}_{\mathrm{red}}x + \mathcal{V}_{\mathrm{null}}\mathcal{V}_{\mathrm{null}}^\dagger X_{\mathrm{Neumann}}\;.
\end{equation}
Notice that there was no need to compute further factorisations of the matrices, and hence, the additional computational cost is just the multiplication of the matrices at the right hand side of Eq.~\eqref{eq:boundarysolution}\,.

Let $r:=\mathrm{rank}(\mathcal{F})\leq 2N_S$ be the rank of the matrix $\mathcal{F}$. This means that the set of linear systems of Eq.~\eqref{eq:linearsystemcompact} provides $r$ linearly independent base-functions at the boundary satisfying the boundary conditions. 
In addition to the solutions above, we need to consider an additional set of linearly independent boundary functions satisfying trivially the boundary conditions, i.e.\ $a_j=0$\,, $n_j=0$ for $1\leq j\leq N_S$\;. This is necessary to provide a complete basis for the FEM. Having into account Eq.~\eqref{eq:normalderivative}, this implies that we can consider another family of boundary functions determined by
\begin{equation}\label{parameters_trivial}
	a_j = 0\,,\quad b_j=\frac{1}{2}c_j - \frac{1}{9}d_j\,,\quad 1\leq j\leq n\,.
\end{equation}
This family is clearly linearly independent from the family obtained as solution of \eqref{eq:linearsystemcompact}. Hence, we have two different families of base-functions at the boundary. The base-functions satisfying Eq.~\eqref{parameters_trivial}, that will be called \emph{1$^{st}$ family} from now on, and the solutions of Eq.~\eqref{eq:linearsystemcompact}, that will be called \emph{2$^{nd}$ family}.

Once the finite element model is determined, we have to compute the entries of the matrices $M$, $B$ and $F$ corresponding to the base-functions at the boundary, cf.\ Eq.~\eqref{eq:geneigenvprob}. Those are, according to Eq.~\eqref{eq:refpolynomial}, as follows:
\begin{equation}\label{base_function_square}
\phi_i=\begin{cases}
p^1_{1i}y^3x+p^1_{2i}y^3+p^1_{3i}y^2x+p^1_{4i}y^2 +p^1_{5i}yx+p^1_{6i}y+p^1_{7i}x+p^1_{8i},&\!\text{square}\:1 \\
p^2_{1i}y^3x+p^2_{2i}y^3+p^2_{3i}y^2x+p^2_{4i}y^2+p^2_{5i}yx+p^2_{6i}y+p^2_{7i}x+p^2_{8i},& \!\text{square}\:2 \\
\hfil\vdots&\hfil\vdots\\
p^{N_S}_{1i}y^3x\hspace{-0.03cm}+\hspace{-0.03cm}p^{N_S}_{2i}y^3\hspace{-0.03cm}+\hspace{-0.03cm}p^{N_S}_{3i}y^2x\hspace{-0.03cm}+\hspace{-0.03cm}p^{N_S}_{4i}y^2\hspace{-0.03cm}+\hspace{-0.03cm}p^{N_S}_{5i}yx\hspace{-0.03cm}+\hspace{-0.03cm}p^{N_S}_{6i}y\hspace{-0.03cm}+\hspace{-0.03cm}p^{N_S}_{7i}x\hspace{-0.03cm}+\hspace{-0.03cm}p^{N_S}_{8i}\!,&
\!\text{square}\: N_S \\
\vspace{-0.4cm}\phantom{a}
\end{cases} 
\end{equation}
where the $x,y$ coordinates correspond to the parallel and normal direction on each square with respect to the boundary.  Notice that the base-functions are cubic in $y$ and linear in $x$, with $i=1,\ldots,N_S+r$, where $N_S$ is the number of squares at the boundary and $r$ is the rank of $\mathcal{F}$. 

The linear transformation that maps the reference square in Fig.\,\ref{Ref_square} to any square is
\begin{equation}
\begin{pmatrix}
x\\
y
\end{pmatrix}=\begin{pmatrix}
x_1\\
y_1
\end{pmatrix}+\begin{pmatrix}
x_2-x_1&0\\
0&y_6-y_1
\end{pmatrix}\begin{pmatrix}
\eta\\
\xi
\end{pmatrix},
\end{equation}
and the first derivatives of the coordinates in the reference square with respect to $x$ and $y$ are
\begin{equation}
\begin{pmatrix}
\eta_x&\hspace{0.5cm}\eta_y\vspace{0.5cm}\\
\xi_x&\hspace{0.5cm} \xi_y
\end{pmatrix}=\begin{pmatrix}
\displaystyle{\frac{1}{x_2-x_1}}& 0\\
0 &\displaystyle{\frac{1}{y_6-y_1}}
\end{pmatrix}.
\end{equation}

Finally, to compute the integrals \eqref{matrices_M_B} and \eqref{matrix_F}, one needs to compute the contributions of each square to the matrix elements, i.e.\ $M_{ij}|_{S_k}$\,, $B_{ij}|_{S_k}$ and $F_{ij}|_{S_k}$\,. Each one of these contributions is the result of the integral of a certain polynomial on the reference square $\mathcal{R}_S$ in terms of the values of the nodes. It is important to remark that the contribution to $F_{ij}|_{S_k}$ has to be done by integrating in the counterclockwise orientation of the boundary.


\subsection{Convergence of the finite element model}

In the previous subsection, we have introduced a finite element model to approximate the spectral problem of Eq.~\eqref{spectralproblem}. It remains to show that the finite element model is indeed an approximating family, cf.\ Def.~\ref{def: approx family}. First of all, let us remark that the condition showed in Eq.~\eqref{eq:h1aprox} is satisfied as a consequence of choosing piecewise continuous functions as finite element model. In order that our family becomes an approximating family, we just need to choose $U^N$ as the unique unitary operator which corresponds to a projector $P_{U^N}$, and a self-adjoint operator $A_{U^N}$ satisfying Assumption \ref{ass:1}. The following result shows a particular simple choice that gathers these conditions.\\

\begin{proposition}
	Let $P_U$ and $A_U$ be the operators defining the spectral problem of Eq.~\eqref{spectralproblem}. Suppose that $P_U$ is a continuous operator with respect to the Sobolev norm $\H^{1/2}(\pO)$\,. Let $K^N$ be the orthogonal projectors onto the subspaces defined by the piecewise linear continuous functions parameterised by the Legendre polynomials of Eq.~\eqref{eq:legendrep} and such that $$\norm{K^N\varphi}_{\H^{1/2}(\pO)}\leq c\norm{\varphi}_{\H^{1/2}(\pO)}\;.$$ Let 
	$$P_{U^N}:=K^NP_{U}K^N\quad \text{and}\quad  A_{U^N}:=K^NA_{U}K^N\,,$$
then, $P_{U^N}$ and $A_{U^N}$ satisfy Assumption \ref{ass:1}.\\
\end{proposition}

\begin{proof}
Let us check first point \ref{en:ass1i}):
\begin{align*}
    \norm{(K^NP_UK^N-P_U)\varphi}_{\H^{1/2}(\pO)} 
        & \leq \norm{(K^N-\mathbb{I})P_U\varphi}_{\H^{1/2}(\pO)} \\
        &\phantom{aaaaaaaaaaaaaaaa}+ \norm{(K^NP_U(\mathbb{I}-K^N)\varphi}_{\H^{1/2}(\pO)}\\
        & \leq \norm{(K^N-\mathbb{I})P_U\varphi}_{\H^{1/2}(\pO)} + c\norm{(\mathbb{I}-K^N)\varphi}_{\H^{1/2}(\pO)}\,,
\end{align*}
where we have used that $P_U$ is continuous with respect to the Sobolev norm $\H^{1/2}(\pO)$\,. Since $K^N$ is the projector onto the space of continuous piecewise linear functions at the boundary and 
$$\norm{(\mathbb{I}-K^N)\varphi}_{\H^{1/2}(\pO)}\leq \norm{(\mathbb{I}-K^N)\varphi}_{\H^{1}(\pO)}\;,$$
the right hand side can be chosen smaller than any $\epsilon$ by choosing the lattice spacing small enough.

Let us now prove \ref{en:ass1ii}). Let $\psi, \varphi \in \L^2(\pO)$, then,
\begin{align*}
    \scalar{\psi}{(A_{U^N}-A_U)\varphi} 
        &= \scalar{\psi-K^N\psi}{(A_{U^N}-A_U)\varphi}\\ 
        &\phantom{aaaaa}+ \scalar{K^N\psi}{(A_{U^N}-A_U)(\varphi - K^N\varphi)} \\
        &\phantom{aaaaa}+ \scalar{K^N\psi}{(A_{U^N}-A_U)K^N\varphi}\\
        &\leq 2\norm{(\mathbb{I}-K^N)\psi}\norm{A_U}\norm{\varphi} + 2\norm{K^N\Psi}\norm{A_U}\norm{(\mathbb{I}-K^N)\varphi}\\
        &\phantom{aaaaa}+ \scalar{\psi}{(A_{U^N}-K^NA_UK^N)\varphi}\,.
\end{align*}
The third factor at the right hand side of the inequality is trivially zero, and the other two go to zero by the same argument of the previous point.
\end{proof}


\section{Numerical examples}\label{sec:examples}

In this section we compute explicitly the solutions of several situations to show the capabilities of the algorithm. This will serve not only for checking the effectiveness of the procedure but also to show how to chose the input matrix $U$ in order to implement boundary conditions for different applications. In all the examples the parameter that fixes the discretisation size is taken to be the number $n$ of divisions of one side of the boundary.

\subsection{Dirichlet and Neumann boundary conditions}\label{ex:DirichNeum}
This first example is the most common and can be easily implemented in FEM. However, since the solution can be computed analytically it serves as a simple check. Moreover, it will serve as a first example showing how to implement the different boundary conditions. It is worth to notice that these two types of boundary conditions have very different nature. Dirichlet boundary conditions are of essential type while Neumann boundary conditions are of natural type. Nevertheless, the numerical scheme can handle both of them on the same footing.
As explained in Section~\ref{FEM:Boundary}, the unitary matrix that serves as input is an approximation to the space of piecewise linear functions with basis the Legendre polynomials at each site of the discretisation of the boundary. The unitary matrix has the block-wise linear structure
\begin{equation}
	U = 
		\begin{bmatrix}
			U_{[00]} & U_{[01]} \\ U_{[10]} & U_{[11]}
		\end{bmatrix}\;,
\end{equation}
where $U_{[00]}$, \dots,$U_{[11]}$ are the blocks corresponding to Eq.~\eqref{eq:unitarylegendre}. If $N_S$ is the number of finite elements at the boundary, the blocks $U_{[ij]}$ are $N_S\times N_S$ and therefore $U$ is of shape $2N_S\times2 N_S$\,. As can be read from Eq.~\eqref{eq:asorey}, Dirichlet and Neumann boundary conditions are almost trivial to implement. The former corresponds to the choice $U = -\mathbb{I}$ and the latter for $U = \mathbb{I}$\,. From Eq.~\eqref{eq:asorey} or its approximated counterpart, Eq.~\eqref{eq:approxbc} and Eq.~\eqref{eq:legendrebc}, it is clear that the former imposes $\varphi = 0$ while the latter imposes $\dot{\varphi} = 0$\,.

In this simple situation, the linear system $\eqref{eq:linearsystemcompact}$ returns the base-functions at the boundary of these well-studied problems and, in both cases, the solutions obtained agree in accuracy with the expected behaviour of a FEM with piecewise linear polynomials, cf.~\cite{brenner08}. Results are shown in Table~\ref{table:DirichletNeumann}

\begin{table}[h]
\begin{tabular}{c|cccccc}
 & Eigenvalue&&&&&\\ 
	 & 1& 2& 3& 4& 5& 6\\ 
\hline 
	Dirichlet & 19.7402& 49.3536& 49.3536& 78.9733& 98.7159& 98.7159\\ 
	Exact & 19.7392&  49.3480& 49.3480& 78.9568& 98.6960& 98.6960\\ 
\hline
	Neumann & -0.0000& 9.8717& 9.8717& 19.7480& 39.4810& 39.4965\\ 
	Exact & 0 & 9.8696& 9.8696& 19.7392& 39.4784& 39.4784\\ 
\end{tabular}
\caption{\small Dirichlet and Neumann eigenvalues for $n = 201$ and the exact solutions respectively.}\label{table:DirichletNeumann}
\end{table}

\subsection{Robin boundary conditions}\label{ex:Robin}
This is an example of a situation where natural boundary conditions are considered. These boundary conditions for the Laplace operator are written as
$$\dot{\varphi} = \Lambda \varphi\;.$$
In the context of solid state physics these boundary conditions represent the interphase between a superconductor, sitting in the region $\Omega$, surrounded by an insulator \cite{asorey13,asorey16}. Applications of these boundary conditions go far beyond the applications in Quantum Physics and they can be applied to a large variety of physical problems like heat conduction, where they are known as convective boundary conditions \cite{HO12} or electromagnetic problems, where they are related with the impedance boundary conditions \cite{mohsen82}.
Notice that the solutions will not be calculated in the standard way, that is, using Neumann boundary conditions, i.e.\ $U = \mathbb{I}$, and then calculating the matrix $F$, cf. Eq.~\eqref{matrix_F} and Eq.~\eqref{eq:geneigenvprob}, associated to them by calculating the scalar products $\scalar{\varphi}{\Lambda\varphi}_{\pO}$\,. Instead, we will do it by considering the unitary matrix 
\begin{equation}\label{eq:robinmatrix}
	U_{kl} = e^{i\alpha}\cdot\delta_{kl}
\end{equation}
as input for the problem. Let us justify the choice of this unitary matrix. Consider the unitary operator on $\L^2(\pO)$ defined by:
\begin{equation}\label{eq:unitaryrobin}
	U\varphi = e^{i\alpha}\varphi\;.
\end{equation}
Simple algebraic manipulations on Eq.~\eqref{eq:asorey} show that this is equivalent to 
\begin{equation}
	\dot{\varphi}=-\tan(\frac{\alpha}{2})\varphi\,,\quad \text{with}\quad \alpha \neq \pm \pi\;. 
\end{equation}
The case $\alpha = \pm \pi$, corresponds to $U = -\mathbb{I}$\,, i.e.\ Dirichlet boundary conditions. The Robin parameter is therefore $\Lambda=-\tan(\frac{\alpha}{2})$. Since this unitary operator is a constant multiplication operator, its representation in the Legendre basis expansion is again a constant multiplication operator, thus, we get \eqref{eq:robinmatrix}.

Notice that the behaviour for positive and negative values of $\alpha$ is very different. While $\alpha > 0$ leads to positive definite self-adjoint extensions of the Laplace operator, for $\alpha < 0$ the problem is no longer positive defined and negative eigenvalues appear. Nevertheless, it is still lower semi-bounded, cf.\ \cite{grubb12, ibortlledo14b,asorey13}, and thus the convergence is ensured. As already happens in the one-dimensional version of this procedure, cf.\ \cite{Ib13}, convergence for the $\alpha<0$ case is slower than for the $\alpha>0$ situation. This is expected to happen because the eigenvalue problem for Robin boundary conditions with positive parameter presents locking, cf. \cite{babuska92}, if the spectrum of the unitary operator is close to $\{-1\}$, i.e.\ $\alpha \to -\pi$. In Table \ref{table:robin-} and Table \ref{table:robin}, are shown the calculated eigenvalues for $\alpha = -0.9 \pi$ and for $\alpha = 0.9 \pi$, respectively, for different discretisation sizes.\\

\begin{table}[h]
\begin{tabular}{c|cccccc}
 & Eigenvalue&&&&&\\ 
	 n& 1& 2& 3& 4& 9& 10\\ 
\hline 
	101& -64.7361& -63.4717& -63.4717& -62.1488& 28.7944& 31.6599\\ 
	153& -68.8114& -67.7815& -67.7815& -66.7229& 27.8643& 29.6896\\ 
	201& -71.0235& -70.1025& -70.1025& -69.1627& 27.4320& 28.7390\\ 
	251& -72.5633& -71.7112& -71.7112& -70.8456& 27.1572& 28.1208\\ 
\end{tabular}
\caption{\small Robin eigenvalues for $\alpha = -0.9 \pi$ for discretisation sizes $n = 101, 153, 201, 251$.}\label{table:robin-}
\end{table}
\vspace{-0.5cm}
\begin{table}[h]
\begin{tabular}{c|cccccc}
 & Eigenvalue&&&&&\\ 
	 n& 1& 2& 3& 4& 5& 6\\ 
\hline 
	101& 11.7493& 30.5681& 30.5681& 49.2777& 64.7353& 64.9841\\ 
	153& 11.7108& 30.4944& 30.4944& 49.2079& 64.6621& 64.8201\\ 
	201& 11.6927& 30.4598& 30.4598& 49.1745& 64.6283& 64.7460\\ 
	251& 11.6811& 30.4377& 30.4377& 49.1528& 64.6067& 64.6996\\ 
\end{tabular}
\caption{\small Robin eigenvalues for $\alpha = 0.9 \pi$ for discretisation sizes $n = 101, 153, 201, 251$.}\label{table:robin}
\end{table}

\subsection{Periodic boundary conditions}\label{ex:Periodic}
In this example we consider periodic boundary conditions. These constitute an example of purely essential boundary conditions.  The convenience of choosing the Legendre basis expansion at each interval of the boundary will be manifest with this example. Let us introduce first the unitary operator that implements the exact boundary conditions on the square $[0,1]\times[0,1]$. We will impose periodic boundary conditions along opposed sites of the square. The unitary operator that corresponds to periodic boundary conditions has a block-wise structure according to the decomposition of the Hilbert space at the boundary corresponding to the four sides of the square depicted in Fig. \ref{fig:periodicbc}, cf. \cite[Example 5.2]{ibortlledo14b},
\begin{equation}\label{eq:hilbertboundarydecomp}
	\L^2(\pO) = \L^2(I_1)\oplus\L^2(I_2)\oplus\L^2(I_3)\oplus\L^2(I_4)\;.
\end{equation}
Here, $\L^2(I_i) = \L^2(I_j), i,j = 1,\dots,4$\,.

\begin{figure}[h]
\centering
\includegraphics[width=4.5cm]{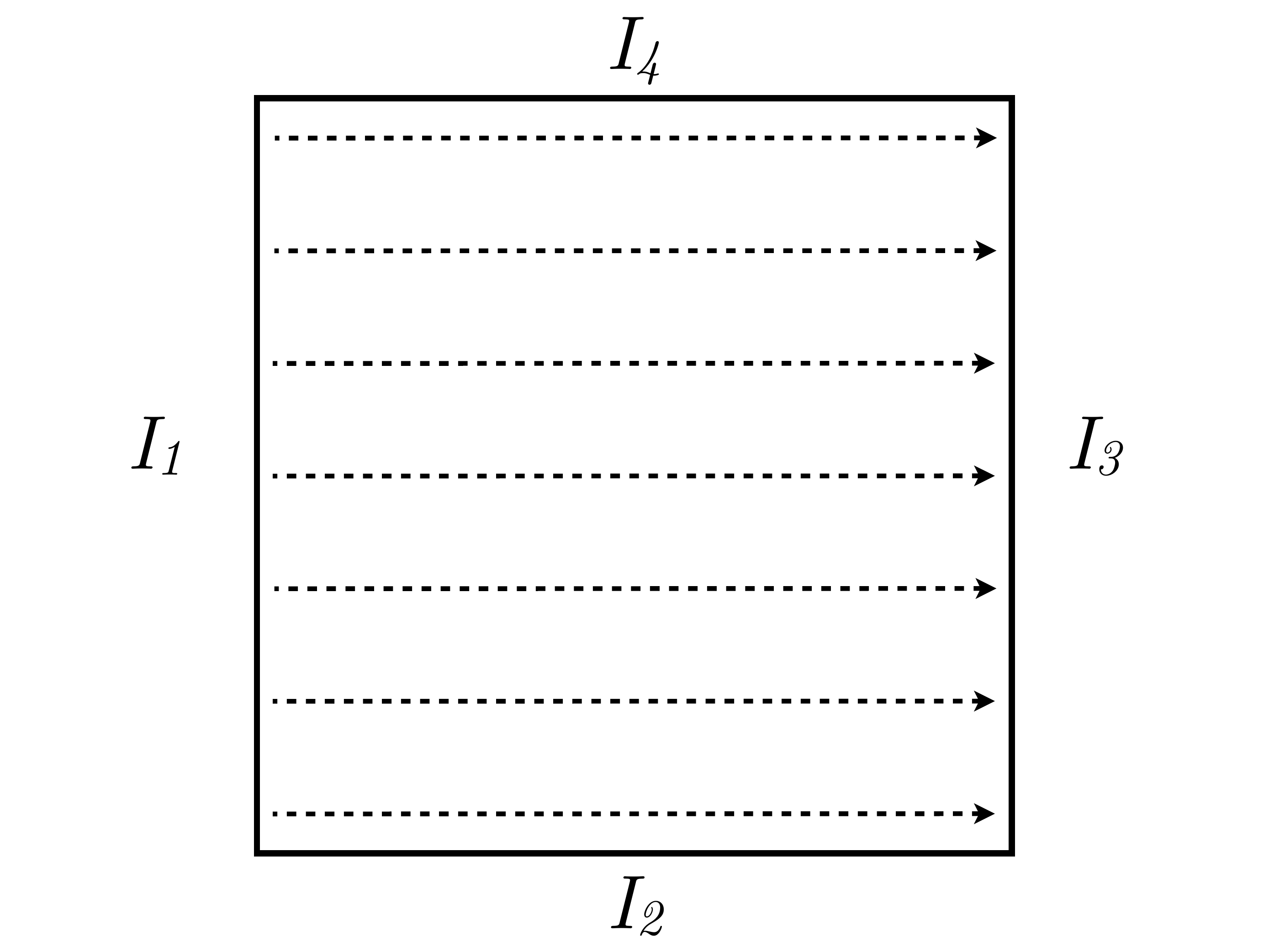}
\captionsetup{margin=2cm}
\caption{\small Identification of the opposite sites of the square that gives rise to periodic boundary conditions. Only the identification between $I_1$ and $I_3$ is depicted.}
\label{fig:periodicbc}
\end{figure}

Notice that a direct identification $I_1 = I_3$ and $I_2 = I_4$ would not let to periodic boundary conditions since both intervals have the same orientation. On the contrary, we want the left end of $I_1$ to be identified with the right end of $I_3$, and equivalently for $I_2$ and $I_4$. This implies that, in addition to this identification, we need to use an orientation reversing diffeomorphism. In coordinates relative to the centre of each interval, the latter takes the form $T:x \to -x$\,. Therefore, in the block-wise structure coming from the decomposition \eqref{eq:hilbertboundarydecomp}, the unitary operator takes the form
\begin{equation}\label{eq:periodunitaryexact}
	U = 
		\begin{bmatrix}
			0 & 0 & T^* & 0 \\
			0 & T^* & 0 & 0 \\
			T^* & 0 & 0 & 0 \\
			0 & 0 & 0 & T^* 
		\end{bmatrix}\;,
\end{equation}
where $T^*$ denotes the pull-back under the diffeomorphism. This is the unitary operator implementing the exact boundary condition of the problem. Inserting this unitary operator in Eq.~\eqref{eq:asorey}, and after some manipulations, it is easy to obtain that this is equivalent to $$\gamma(\Phi)|_{I_1} = T^*\gamma(\Phi)|_{I_3}$$ and similarly for $I_2$ and $I_4$\,.

Now, we need to express this unitary operator in the basis of the Legendre polynomials, see Eqs.~\eqref{eq:unitarylegendre}. Assume first that we take the coarsest FE approximation of the boundary that is possible in this setting. That is, the discretisation given by the four intervals $I_i$, $i= 1,\dots,4$\,. For this coarse approximation, the matrix $U_{kl}$ is an $8\times8$ unitary matrix. Since the basis for the finite elements approximation at the boundary is given in terms of Legendre polynomials, most of the terms cancel out except:
\begin{align*}
	\scalar{P_0^1}{UP_0^3}  &=  \scalar{P_0^3}{UP_0^1} =\scalar{P_0^2}{UP_0^4}  =  \scalar{P_0^4}{UP_0^2} =1\,,\\
	\scalar{P_1^1}{UP_1^3}  &=  \scalar{P_1^3}{UP_1^1} = \scalar{P_1^2}{UP_1^4}  =  \scalar{P_1^4}{UP_1^2} = -1\,.
\end{align*}
The difference in sign comes from the fact that the orientation reversing diffeomorphism leaves the even order Legendre polynomials invariant while the odd order polynomials change sign. Extending this result to the general situation, where each side of the square is divided into $n$ intervals, is straightforward. The unitary matrix $U_{kl}$ takes the form
\begin{align}
	U_{kl} &= 
		\begin{bmatrix}
			U_0 & 0 \\
			0 & -U_0
		\end{bmatrix}
		\in \mathbb{C}^{8n\times 8n}\;,\quad 
	U_0 = 
		\begin{bmatrix}
			0 & 0 & P & 0 \\
			0 & P & 0 & 0 \\
			P & 0 & 0 & 0 \\
			0 & 0 & 0 & P
		\end{bmatrix}
		\in \mathbb{C}^{4n\times 4n}\;, \notag\\
	&\hspace{2.3cm} P =
		\begin{bmatrix}
			0 & & 1\\
			   &\udots & \\
			 1 & & 0
		\end{bmatrix}
		\in \mathbb{C}^{n\times n}\;.
\end{align}

In Table~\ref{table:periodic} are shown the lowest eigenvalues with the aforementioned periodic boundary conditions for several values of the discretisation parameter and the exact values computed analytically.

\begin{table}[h]
\begin{tabular}{c|cccccc}
 & Eigenvalue&&&&&\\ 
	 n& 1& 2& 3& 4& 5& 6\\ 
\hline 
	101& -0.0000& 39.4885& 39.5103& 39.5103& 39.5487& 79.0194\\ 
	153& 0.0000& 39.4829& 39.4923& 39.4924& 39.5094& 78.9849\\ 
	201& -0.0000& 39.4810& 39.4865& 39.4866& 39.4965& 78.9733\\ 
	251& 0.0000& 39.4801& 39.4836& 39.4837& 39.4900& 78.9675\\ 
	Exact & 0.0& 39.4784& 39.4784& 39.4784& 39.4784& 78.9568\\ 
\end{tabular}
\caption{\small Eigenvalues for periodic boundary conditions for discretisation sizes $n = 101, 153, 201, 251$.}\label{table:periodic}
\end{table}

\subsection{Quasiperiodic boundary conditions}\label{ex:Quasiperiodic}

In this example we show how to implement quasiperiodic boundary conditions (also called Bloch-periodic boundary conditions). These boundary conditions allow to compute the energy bands of electrons in metals \cite{kittel96,sukumar09}. The domain $\Omega$ represents in this case the fundamental domain of the Crystal. The phase $\alpha$ becomes the momentum on the reciprocal lattice. In general, one needs to include a potential $(-\Delta + V)\Psi = E\Psi$, i.e. to consider the Shr\"odinger equation, to treat this case. As mentioned in the introduction, regular potentials do not change the domains of self-adjointness \cite{reed75} and can be included straightforwardly as in a standard FEM. That is, the matrix elements would be computed by $\scalar{\Phi_i}{V\Phi_j}$ and the implementation of the boundary conditions can be done as introduced in Section \ref{sec:ArbitraryDimension} and Section \ref{sec:FEM}. These boundary conditions appear also in the computation of band gaps in photonic crystal fibres \cite{norton08}. 

We will consider a situation in which we impose quasiperiodic boundary conditions between the sides $I_1$ and $I_3$ of the previous example and periodic boundary conditions between $I_2$ and $I_4$\,. The unitary operator that implements the exact boundary condition is
\begin{equation}\label{eq:quasiperiodunitaryexact}
	U = 
		\begin{bmatrix}
			0 & 0 & e^{i\alpha}T^* & 0 \\
			0 & T^* & 0 & 0 \\
			e^{-i\alpha}T^* & 0 & 0 & 0 \\
			0 & 0 & 0 & T^* 
		\end{bmatrix}\;,
\end{equation}
where $e^{i\alpha}$ is the complex phase between the sites $I_1$ and $I_3$\,, cf. \cite[Example 5.3]{ibortlledo14b}. This is, it suffices to multiply by a complex phase and its conjugate the appropriate rows of the unitary operator implementing periodic boundary conditions defined in the previous example. Similar calculations to those in the previous example lead to the unitary matrix to be used as input in the problem
\begin{align}
	U &= 
		\begin{bmatrix}
			U_0 & 0 \\
			0 & -U_0
		\end{bmatrix}
		\in \mathbb{C}^{8n\times 8n}\;,\quad 
	U_0 = 
		\begin{bmatrix}
			0 & 0 & e^{i\alpha}P & 0 \\
			0 & P & 0 & 0 \\
			e^{-i\alpha}P & 0 & 0 & 0 \\
			0 & 0 & 0 & P
		\end{bmatrix}
		\in \mathbb{C}^{4n\times 4n}\;, \notag
\end{align}
\begin{align}		
	P &= 
		\begin{bmatrix}
			0 & & 1\\
			   &\udots & \\
			 1 & & 0
		\end{bmatrix}
		\in \mathbb{C}^{n\times n}\;.
\end{align}

In Table~\ref{table:quasiperiodic} are shown the lowest eigenvalues for the case $\alpha = \frac{\pi}{4}$ and the exact analytical values of the problem, which are given by the formula 
$$E = 4\pi^2\left(m^2+\left(n+\frac{1}{8}\right)^2\right)\,,\quad m,n \in\mathbb{Z}\;.$$
One of the eigenfunctions is shown in Fig.~\ref{fig:qausiperiodic}.

\begin{table}[h]
\begin{tabular}{c|cccccc}
 & Eigenvalue&&&&&\\ 
	 n& 1& 2& 3& 4& 5& 6\\ 
\hline 
	101& 0.6173& 30.2514& 40.1278& 40.1370& 50.0140& 69.7748\\ 
	151& 0.6170& 30.2373& 40.1099& 40.1142& 49.9870& 69.7363\\ 
	201& 0.6170& 30.2323& 40.1036& 40.1060& 49.9774& 69.7224\\ 
	251& 0.6169& 30.2299& 40.1006& 40.1022& 49.9730& 69.7159\\ 
	Exact & 0.6168& 30.2256& 40.0952& 40.0952& 49.9648& 69.7040\\ 
\end{tabular}
\caption{\small Eigenvalues for quasi-periodic boundary conditions for $\alpha=\frac{\pi}{4}$ for discretisation sizes $n = 101, 151, 201, 251$.}\label{table:quasiperiodic}
\end{table}

\begin{figure*}[h!]
        \center
        \includegraphics[width=11cm]{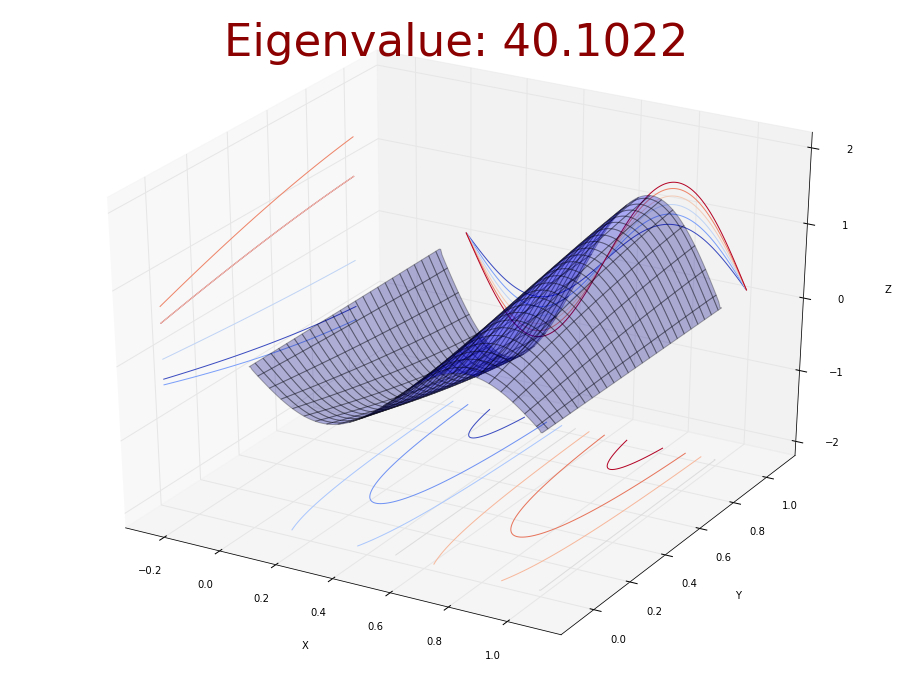}
	\caption{\small Eigenfunction corresponding to the 4$^{\text{th}}$ eigenvalue for quasi-periodic boundary conditions with $\alpha = \frac{\pi}{4}$. Only the real part is plotted.}
	\label{fig:qausiperiodic}
\end{figure*}

Since this is a relevant example for applications we show a convergence test for the fundamental level. For different values of the discretisation we have computed the norm of the difference between the fundamental state computed numerically and the analytical solution, i.e. $\norm{\Psi^N - \Psi}$. In a finite element model with linear polynomials, as is the case, this error is expected to decrease proportionally to the diameter, $\mathcal{K}$, of the finite elements, cf.\ \cite{brenner08}. Notice that $n$ is the number of elements along one side of the boundary and therefore the diameter of the finite elements is $\mathcal{K} \propto \frac{1}{n}$. The results are plotted in Fig.~\ref{fig:error} where it can be checked that the error follows the expected behaviour.

\begin{figure*}[h!]
        \center
	\includegraphics[width=11cm]{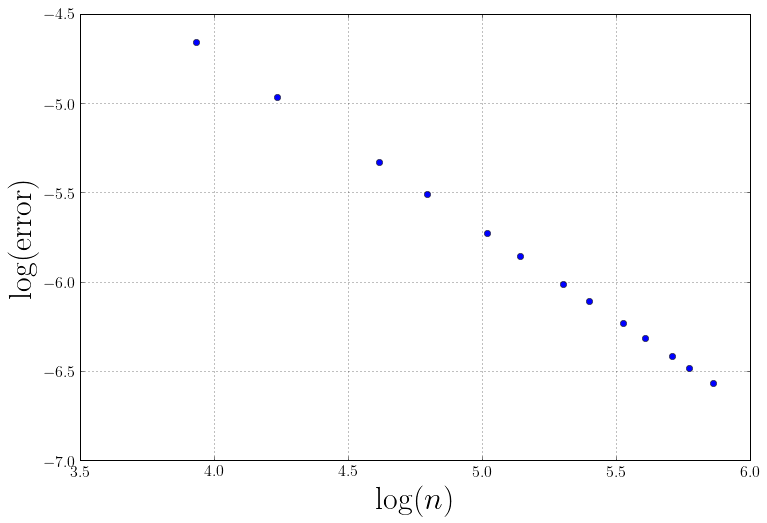}
	\caption{\small Logarithm of the error $\norm{\Psi^N - \Psi}$ of the fundamental state against the logarithm of the discretisation size for the values $n= 51, 69, 101, 121, 151, 171, 201, 221,$ $251, 273, 301, 321, 351.$  The numerical convergence test shows that  the error is proportional to $1/n$. }
	\label{fig:error}
\end{figure*}

\subsection{Discontinuous Robin boundary conditions}\label{ex:DiscontinuousRobin}

In this final example we will consider a situation with less trivial natural boundary conditions. We will consider piecewise constant Robin boundary conditions. These introduce discontinuities at the boundary at the points where the Robin parameter changes. The weak formulation of the Laplace operator presented here can handle with this singularities without special modifications as the analytical structure of the problem is well-defined even in this discontinuous situation.

The physical applications described by these boundary conditions are the same than in Section \ref{ex:Robin}. But in this case, since we are considering discontinuities at the boundary, this would represent that the region $\Omega$ is surrounded by distinct materials. 
This has applications in superconducting circuits where one has to consider that the superconductor is in contact with Josephson junctions and also with the substrate, which is an important source of decoherence \cite{cw08, mar05}.

The unitary operator in this case is almost identical to the one appearing in Section~\ref{ex:Robin} with the exception that for half of the boundary we have to take  a value of the Robin parameter $\Lambda_1$ different from the Robin parameter $\Lambda_2$ on the other half. The singularities are placed in the middle of the intervals $I_2$ and $I_4$\,. In Table~\ref{table:discness} the results for $\Lambda_1 = -\Lambda_2 = 1$ are shown.

\begin{table}[h]
\begin{tabular}{c|cccccc}
 & Eigenvalue&&&&&\\ 
	 n& 1& 2& 3& 4& 5& 6\\ 
\hline 
	101& -1.3236& 7.8229& 10.7855& 22.5004& 36.9358& 42.5570\\ 
	153& -1.4483& 7.5909& 10.7521& 22.4242& 36.7601& 42.3032\\ 
	201& -1.5149& 7.4705& 10.7339& 22.3796& 36.6794& 42.1556\\ 
	251& -1.5621& 7.3869& 10.7208& 22.3460& 36.6276& 42.0448\\ 
	301& -1.5966& 7.3264& 10.7110& 22.3199& 36.5923& 41.9597\\ 
	351& -1.6233& 7.2799& 10.7034& 22.2988& 36.5664& 41.8913\\ 
	401& -1.6448& 7.2428& 10.6972& 22.2812& 36.5465& 41.8344\\ 
\end{tabular}
\caption{\small Eigenvalues for discontinuous Robin boundary conditions for discretisation sizes $n = 101, 153, 201, 251, 301, 351, 401$.}\label{table:discness}
\end{table}

Since the exact solution cannot be obtained analytically, we have computed the approximation of the same problem with the standard approach for natural boundary conditions. That is, we have taken the basis of the Neumann problem and have computed the contribution of the matrix $F$, Eq.~\eqref{eq:geneigenvprob}, as 
$$\scalar{\varphi(x)}{\Lambda(x)\varphi(x)}\;.$$ Here, $\varphi(x)$ are the piecewise linear continuous functions that correspond to the restriction to the boundary of the Neumann boundary functions. The results are presented in Table~\ref{table:discnnat}.

\begin{table}[h]
\begin{tabular}{c|cccccc}
 & Eigenvalue&&&&&\\ 
	 n& 1& 2& 3& 4& 5& 6\\ 
\hline 
	101& -2.2909& 6.2365& 10.4792& 21.4133& 36.1646& 39.2590\\ 
	151& -2.3361& 6.1410& 10.4683& 21.4028& 36.0637& 39.2186\\ 
	201& -2.3597& 6.0923& 10.4630& 21.3992& 36.0152& 39.2014\\ 
	251& -2.3742& 6.0628& 10.4598& 21.3975& 35.9867& 39.1921\\ 
	301& -2.3840& 6.0430& 10.4576& 21.3966& 35.9680& 39.1862\\ 
	351& -2.3911& 6.0287& 10.4561& 21.3960& 35.9549& 39.1822\\ 
	401& -2.3964& 6.0180& 10.4549& 21.3956& 35.9450& 39.1794\\ 
\end{tabular}
\caption{\small Eigenvalues for discontinuous Robin boundary conditions for discretisation sizes $n = 101, 151, 201, 251, 301, 351, 401$ computed with the standard FEM for natural boundary conditions.}\label{table:discnnat}
\end{table}

As it can be seen, both methods converge. In both cases, the convergence of the negative eigenvalues is slow as is to be expected. However, the results do not converge to the same values. Although the computations for the standard approach converge slightly faster, we believe that the results of our approach, showed in Table \ref{table:discness}, are more trustable. In the algorithm introduced in this article, the boundary conditions are implemented exactly and the operator $A_{U}^N$ is the operator $A_U$ of the exact problem. Therefore, the convergence to the solutions of the exact problem is granted as shown in Section~\ref{sec:ArbitraryDimension}. On the contrary, using the standard approach to treat natural boundary conditions, the problem fails to be approximated exactly at the discontinuities. Indeed, at the discontinuities, the Robin parameter is approximated by a linear polynomial with slope $\frac{2}{h}$, being $h$ the step of the discretisation. Hence, the function $\Lambda(x)$ does not converge in the Sobolev norm of order 1 to the exact Robin function. Of course, the situation treated here is known to be singular and there are known ways to choose the boundary element functions in order to improve the standard approach, cf. \cite{strouboulis00} and references therein, as well as using mesh refinement around the discontinuity points. Our algorithm presents the advantage that it chooses implicitly appropriate base-functions at the boundary in the absence of an \emph{a priori} estimation of the discretisation or of the finite element basis.\\

\section{Conclusions}

A numerical scheme has been designed to implement a wide variety of boundary conditions to solve the spectral problem for the Laplace-Beltrami operator. The application of these boundary conditions is not limited to the Laplace-Beltrami operator and can be applied straightforwardly to other operators relevant to it like Schr\"odinger operators. 

The convergence of the scheme is proven rigorously for any dimension and is implemented in dimension two to show its capabilities. It is shown that it has applications in current and relevant physical problems. 

This scheme is able to deal simultaneously with essential and natural boundary conditions with no \emph{a priori} information and it is shown that it has better performance than the usual approach to deal with natural boundary conditions in the presence of discontinuities. The latter case is also relevant for physical applications as discussed in Section \ref{ex:DiscontinuousRobin}.

Future work will be devoted to implement more efficient versions of the scheme with faster convergence. For instance, admitting higher order degrees of the finite element model and variable geometries of the region $\Omega$. This will be done by complementing current available software with the structure of the boundary elements developed in this work. Other important development of this numerical scheme will be its adaptation to deal with other relevant operators in Physics like the Dirac operator. This would allow, for instance, to compute the spectral problem of energy excitations in Graphene layers \cite{SAHR11} or in the Quantum Spin Hall effect \cite{BZ06}. There are self-adjoint boundary conditions like the ones considered in \cite{CP17} or \cite{Pe17}, that are closely related to the situation described in Section \ref{ex:DiscontinuousRobin}, which are relevant for physical applications but that cannot be treated straightforwardly with the present approach. These types of boundary conditions will be considered in future research.

\newpage
\appendix
\section{The pseudocode}\label{pseudocode_sec}

\begin{figure}[!htb]
\centering
\includegraphics[width=7cm]{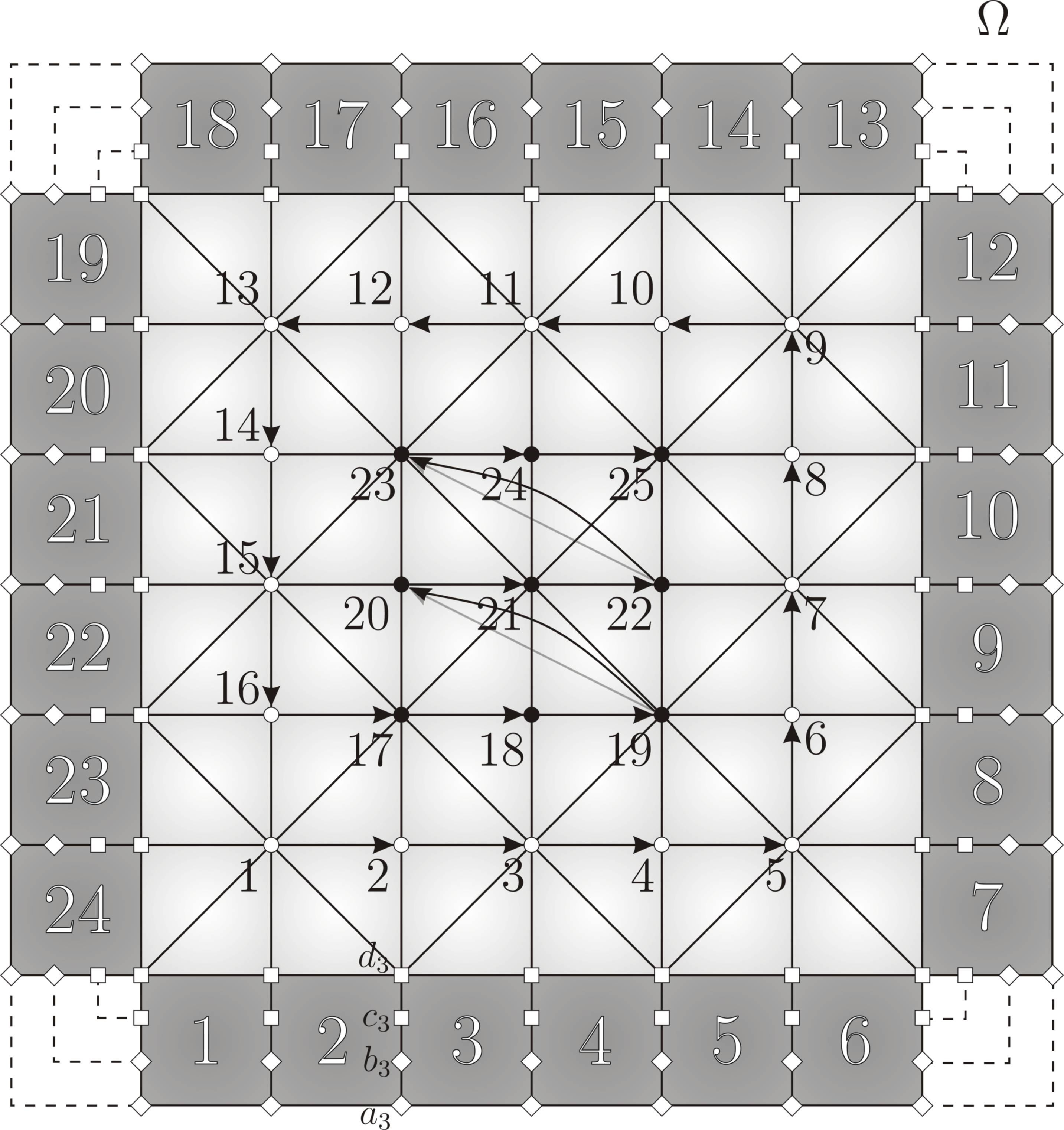}
\caption{\small Sorting of nodes at the bulk and at the boundary. Notice that odd nodes at the bulk correspond to star-like nodes and even nodes correspond to diamond-like nodes, see Fig.~\ref{two_nodes}.}
\label{sort:bulk}
\end{figure}

In this appendix, we show how to construct the algorithm described in section \ref{sec:FEM} in a schematic way. First of all, let us present the inputs and outputs of the program.
\medskip

\textbf{Inputs}:
\begin{itemize}[leftmargin=0.5in]

\item $n$, number of nodes in which both axis are discretised (remember that we have considered $n=m$ for simplicity and must be odd to obtain the triangularisation in Fig.\,\ref{Reticulo_ext}).

\item $U$, unitary matrix of size $2N_S\times 2N_S$ that contains the information of the boundary conditions satisfying Eq.~\eqref{eq:legendrebc}.

\end{itemize}

\textbf{Outputs}:
\begin{itemize}[leftmargin=0.5in]

\item $\phi_i$ and $u_i$, $i=1,\ldots,N$, $N=N_B+2N_S+r$, base-functions and coefficients of the expansion:
\begin{equation}\label{expansion_phi}
\Phi_N=\sum_{i=1}^{N}u_i\phi_i,
\end{equation}

\item $\lambda_i$, $i=1,\ldots,N$, eigenvalues of the spectral problem:
\begin{equation}
(M-F)u_i=\lambda_iBu_i,\qquad i=1,\ldots,N.
\end{equation}

\end{itemize}

Following, we write the pseudocode of the program we have designed based in MATLAB:\\
\\
\noindent{\color{line1}\textbf{function}} $\boldsymbol{[}\phi_i,u_i,\lambda_i\boldsymbol{]}$ = \textbf{SolFEM} $\boldsymbol{(}n,U\boldsymbol{)}$
	\vspace{-0.1cm}\begin{siderules}[leftmargin=0.2cm, linecolor=line1]

		\begin{itemize}[leftmargin=0.225in]

			\item[{\color{line2}$\diamond$}]\textbf{Base-functions at the boundary.}
				\vspace{-0.15cm}\begin{siderules}[leftmargin=-0.28cm, linecolor=line2]

					\vspace{-0.1cm}\begin{itemize}[leftmargin=0.225in]

						\item[$\bullet$] $N_S=4(n-1)$\,;\; $\displaystyle{h=\frac{1}{n+1}}$\,;

						\item[$\bullet$] Compute matrices $\mathcal{F}_0$, $\mathcal{F}_1$, $\mathcal{C}_0$ and $\mathcal{C}_1$ by means of Eqs.~\eqref{eq:linearsystemcomplete}.

						\item[$\bullet$] $\displaystyle{\mathcal{F}=\left[
							\begin{matrix}
\mathcal{F}_0\\
\mathcal{F}_1
\end{matrix}\right]}$\,;\; $\displaystyle{\mathcal{C}=\left[
\begin{matrix}
\mathcal{C}_0\\
\mathcal{C}_1
							\end{matrix}\right]}$\,;

						\item[$\bullet$] Compute the SVD of matrix $\mathcal{F}$, $\mathcal{F}=\mathcal{W}\mathcal{S}\mathcal{V}^\dagger$.
							\begin{itemize}[leftmargin=0.2in]
								{\color{mygreen}\vspace{-0cm}\item[$\#$]Notice that $\mathcal{S}=\text{diag}([s_1,\ldots, s_{2N_S}])$.}
							\end{itemize}
							
							\item[$\bullet$] $tol=2N_S\,\text{eps}\left(\text{max}(\text{diag}(\mathcal{S}))\right)$\,;
							
							\begin{itemize}[leftmargin=0.2in]
								{\color{mygreen}\vspace{-0cm}\item[$\#$]Default choice of tolerance in the rank computation of a matrix in MATLAB.}
								
								{\color{mygreen}\vspace{-0cm}\item[$\#$]$\text{eps}(x)$ is the distance between $|x|$ and the next larger number.}
								
							\end{itemize}
							
						\item[$\bullet$]$r=\text{sum}(\mathrm{diag}(\mathcal{S})>tol)$\,;
							


							\begin{itemize}[leftmargin=0.2in]
								{\color{mygreen}\vspace{-0cm}\item[$\#$]$r=\mathrm{rank}(\mathcal{F})$.}
							\end{itemize}

						\item[$\bullet$]$\mathcal{S}_{\mathrm{red}}=\mathcal{S}(1:r\,,\,1:r)$\,;
						\item[$\bullet$]$\mathcal{W}_\mathrm{red}=\mathcal{W}(\,:\,,\,1:r)$\,;\,$\mathcal{W}_\mathrm{null}=\mathcal{W}(\,:\,,\,r+1:2N_S)$\,;
							\begin{itemize}[leftmargin=0.2in]
								{\color{mygreen}\vspace{-0cm}\item[$\#$] Check if the system is compatible.}
							\end{itemize}

						\item[{\color{line3}\textbf{if}}] $\|\mathcal{W}_\mathrm{null}^\dagger\mathcal{C}\|_2> \,tol$
							\begin{siderules}[leftmargin=-0.25cm, linecolor=line3]
								\vspace{0cm}\begin{itemize}[leftmargin=0.225in]

									\item[$\bullet$] \textbf{error}(\,{\color{myviolet}\textbf{`System is not compatible'}}\,)\,;
%

								\end{itemize}

							\end{siderules}

						\hspace{-0.5cm}{\color{line3}\textbf{end}}

						\item[$\bullet$] Solve the system $\mathcal{S}_{\mathrm{red}}x = \mathcal{W}_{\mathrm{red}}^\dagger\mathcal{C}(\,:\,,\,1:r)$\,.

						\item[$\bullet$] Write $X_{\mathrm{Neumann}}$ as in Eq.~\eqref{Neumann_matrices_alg}.

						\item[$\bullet$] $\mathcal{V}_{\mathrm{red}}=\mathcal{V}(\,:\,,\,1:r)$\,;\;$\mathcal{V}_{\mathrm{null}}=\mathcal{V}(\,:\,,\,r+1:2N_S)$\,;

						\item[$\bullet$] $s = \mathcal{V}_{\mathrm{red}}x + \mathcal{V}_{\mathrm{null}}\mathcal{V}_{\mathrm{null}}^\dagger X_{\mathrm{Neumann}}(\,:\,,\,1:r)$\,;
						\begin{itemize}[leftmargin=0.2in]
							{\color{mygreen}\vspace{-0cm}\item[$\#$]Solution that minimises \eqref{minimise_s}.}
						\end{itemize}
						
						\item[{\color{line3}$\diamond$}] \textbf{Trivial base-functions at the boundary.}
						
							\vspace{-0cm}\begin{siderules}[leftmargin=-0.29cm, linecolor=line3]
							
								\begin{itemize}[leftmargin=0.2in]
									\item[{\color{mygreen}$\#$}]{\color{mygreen}Parameters for the $2N_S$ base-functions satisfying trivial boundary conditions \eqref{parameters_trivial}.}
									
									\item[{\color{mygreen}$\#$}]{\color{mygreen}$k=1,\ldots,N$ is the label of the squares at the boundary, as depicted in Fig.\,\ref{sort:bulk} with white numbers, \,and \,$j=1,\ldots,2N_S$}
								\end{itemize}
														
							\end{siderules}

					\end{itemize}
					
					\phantom{a}

				\end{siderules}

		\end{itemize}
		
		\phantom{a}

	\end{siderules}
\newpage	
			
	\vspace{0.0cm}\begin{siderules}[leftmargin=0.2cm, linecolor=line1]
	
	\hspace{0.2cm}{\color{line2}$\diamond$}
	\vspace{0.25cm}\phantom{a}
		\begin{itemize}[leftmargin=0.225in]

			\item[\phantom{a}]
				\vspace{-0.2cm}\begin{siderules}[leftmargin=-0.28cm, linecolor=line2]
				
				\hspace{0.2cm}{\color{line3}$\diamond$}
				
					\begin{itemize}[leftmargin=0.225in]
					
					\item[\phantom{a}]
						
							\vspace{-0.1cm}\begin{siderules}[leftmargin=-0.29cm, linecolor=line3]
							
								\begin{itemize}[leftmargin=0.2in]
									
									\item[{\color{mygreen}\phantom{$\#$}}]{\color{mygreen}is the label of the base-functions at the boundary.}

								\end{itemize}
								
								\hspace{-0cm}{\color{line4}\textbf{for}} $j=1:N_S$
							
								\vspace{-0cm}\begin{siderules}[leftmargin=0.29cm, linecolor=line4]
								
								\hspace{-0cm}{\color{line5}\textbf{for}} $k=1:N_S$
								
									\vspace{-0cm}\begin{siderules}[leftmargin=0.29cm, linecolor=line5]
									
										\vspace{-0.1cm}\begin{itemize}[leftmargin=0.42cm]
											
											\item[$\bullet$]$a^j_k=0$\,;\; $\displaystyle{b^j_k=\frac{1}{2}\delta^j_k}$\,;\; $c^j_k=\delta^j_k$\,;\; $d^j_k=0$\,;
											\item[$\bullet$]$a^{j+N_S}_k=0$\,;\; $\displaystyle{b^{j+N_S}_k=-\frac{1}{9}\delta^j_k}$\,;\; $c^{j+N_S}_k=0$\,;
											\item[$\bullet$]$d^{j+N_S}_k=\delta^j_k$\,;

										\end{itemize}
								
									\end{siderules}
									
									\vspace{-0.1cm}{\color{line5}\textbf{end}}
								
								\end{siderules}
								
								{\color{line4}\textbf{end}}
							
							\end{siderules}
							
							\vspace{-0.1cm}\hspace{-0.6cm}{\color{line3}\textbf{end}}
							
							\item[$\bullet$]$\displaystyle{V=\frac{1}{2}\left(\begin{matrix}
							\phantom{-}9 & -9 & \phantom{-}27 & -27 & \phantom{-}9 & -9 &\phantom{-}27 & -27\\
							-9 & \phantom{-}0 & \phantom{-}0 & \phantom{-}0 & \phantom{-}0 & \phantom{-}9 & -27 & \phantom{-}27\\
							-18 & \phantom{-}18 & -45 & \phantom{-}36 & -9 & \phantom{-}9 & -36 & \phantom{-}45\\
							 \phantom{-}18 & \phantom{-}0 & \phantom{-}0 & \phantom{-}0 & \phantom{-}0 & -9 & \phantom{-}36 & -45\\
							 \phantom{-}11 & -11 & \phantom{-}18 & -9 & \phantom{-}2 & -2 & \phantom{-}9 & -18\\
							 -11 & \phantom{-}0 & \phantom{-}0 & \phantom{-}0 & \phantom{-}0 & \phantom{-}2 & -9 & \phantom{-}18\\
							 -2 & \phantom{-}2& \phantom{-}0 & \phantom{-}0 & \phantom{-}0 & \phantom{-}0 & \phantom{-}0 & \phantom{-}0\\
							 \phantom{-}2 & \phantom{-}0& \phantom{-}0 & \phantom{-}0 & \phantom{-}0 & \phantom{-}0 & \phantom{-}0 & \phantom{-}0
														\end{matrix}\right)}\,;$
								\begin{itemize}[leftmargin=0.2in]
								
									\vspace{-0cm}\item[{\color{mygreen}$\#$}]{\color{mygreen}This matrix gives the parameters $p_{1j}^k,\ldots,p_{8j}^k$ in \eqref{base_function_square} from parameters $a_k^j, b_k^j,\ldots$ in $\mathcal{R}_S$ (see Fig.\,\ref{Ref_square}).}
								
								\end{itemize}
							
							\hspace{-0.6cm}{\color{line3}\textbf{for}} $j=1:2N_S+r$
							
								\vspace{-0cm}\begin{siderules}[leftmargin=-0.29cm, linecolor=line3]
								
								\hspace{-0cm}{\color{line4}\textbf{for}} $k=1:N_S$
								
									\vspace{-0cm}\begin{siderules}[leftmargin=0.29cm, linecolor=line4]
									
										\begin{itemize}[leftmargin=0.42cm]
										
											\item[$\bullet$]$\left[\begin{matrix}
											p_{1j}^k\!\!\!\!\!&p_{2j}^k\!\!\!\!\!&p_{3j}^k\!\!\!\!\!&p_{4j}^k\!\!\!\!\!&p_{5j}^k\!\!\!\!\!&p_{6j}^k\!\!\!\!\!&p_{7j}^k\!\!\!\!\!&p_{8j}^k
											\end{matrix}\right]^\top$\\
											
											\vspace{-0.2cm}$\phantom{a}\hspace{2.5cm}=V\left[\begin{matrix}
											a_{k}^j\!\!\!\!\!&a_{k+1}^j\!\!\!\!\!&b_{k+1}^j\!\!\!\!\!&c_{k+1}^j\!\!\!\!\!&d_{k+1}^j\!\!\!\!\!&d_{k}^j\!\!\!\!\!&c_{k}^j\!\!\!\!\!&b_{k}^j
											\end{matrix}\right]^\top$;

										\end{itemize}
								
									\end{siderules}
									
									\vspace{-0.1cm}{\color{line4}\textbf{end}}
								
								\end{siderules}
								
								\vspace{-0.1cm}\hspace{-0.6cm}{\color{line3}\textbf{end}}
								
					\end{itemize}
				
				\end{siderules}
				
				\vspace{-0.3cm}\hspace{-0.6cm}{\color{line2}\textbf{end}}
				
			\end{itemize}
			
			\hspace{-0.33cm}{\color{line2}$\diamond$} \,\textbf{Grid and base-functions at the bulk.}
			
				\vspace{-0.1cm}\begin{siderules}[leftmargin=0.29cm, linecolor=line2]
				
					\begin{itemize}[leftmargin=0.225in]
					
						\item[$\bullet$]$N_B=(n-2)^2\,;$
						
						\item[$\bullet$]Create the bulk grid in the square $[h,1-h]\times[h,1-h]$ (light grey
					
					\end{itemize}
				
				\end{siderules}
				
				\vspace{-0.2cm}\phantom{a}
				
		\end{siderules}
		
		\newpage
		
		\vspace{0.0cm}\begin{siderules}[leftmargin=0.2cm, linecolor=line1]
		
			\vspace{0cm}\hspace{0.2cm}{\color{line2}$\diamond$}
			
				\vspace{-0.1cm}\begin{siderules}[leftmargin=0.29cm, linecolor=line2]
				
					\begin{itemize}[leftmargin=0.225in]
					
						\item[\phantom{$\bullet$}]region in Fig.\,\ref{sort:bulk}) with step $h$ in both axis.
									
						\item[$\bullet$]Sort nodes as in Fig.\,\ref{sort:bulk}.
						
						\hspace{-0.58cm}{\color{line3}\textbf{for}} $i=1:N_B$
						
							\vspace{-0.1cm}\begin{siderules}[leftmargin=-0.31cm, linecolor=line3]
							
								\begin{itemize}[leftmargin=0.42cm]
								
									\item[$\bullet$]Compute the piecewise linear base-functions $\phi_i$ at the bulk according to Eqs.~\eqref{def_base_func} and \eqref{eq:basefunctionbulk}.
								
								\end{itemize}
							
							\end{siderules}
							
							\vspace{-0.cm}\hspace{-0.6cm}{\color{line3}\textbf{end}}
					
					\end{itemize}
				
				\end{siderules}
			
				\vspace{-0.15cm}\hspace{-0.54cm}{\color{line2}\textbf{end}}
				
				\hspace{-0.33cm}{\color{line2}$\diamond$} \,\textbf{Matrix elements at the bulk.}
				
					\vspace{-0.1cm}\begin{siderules}[leftmargin=0.29cm, linecolor=line2]
					
						\vspace{-0.1cm}\begin{itemize}[leftmargin=0.2in]
						
							\item[{\color{mygreen}$\#$}]{\color{mygreen}Matrices $M$, $B$ and $F$ are Hermitean, hence we only need to compute the elements in the diagonal and below it.}
						
						\end{itemize}
						
						\vspace{0.1cm}\hspace{-0.58cm}{\color{line3}\textbf{for}} $j=1:N_B$
						
							\vspace{-0.1cm}\begin{siderules}[leftmargin=0.24cm, linecolor=line3]
							
								\hspace{-0cm}{\color{line4}\textbf{for}} $i=j:N_B$
								
									\vspace{-0.1cm}\begin{siderules}[leftmargin=0.24cm, linecolor=line4]
									
										\vspace{-0.1cm}\begin{itemize}[leftmargin=0.225in]
										
											\item[$\bullet$]Compute $M_{ij}$ and $B_{ij}$ by means of Eqs.~\eqref{Bulk_matrices} and \eqref{Bulk_matrices_sum}.
										
										\end{itemize}
									
									\end{siderules}
									
									\vspace{-0.15cm}\hspace{-0.55cm}{\color{line4}\textbf{end}}
							
							\end{siderules}
							
							\vspace{-0.15cm}\hspace{-0.57cm}{\color{line3}\textbf{end}}
					
					\end{siderules}
				
					\vspace{-0.15cm}\hspace{-0.55cm}{\color{line2}\textbf{end}}
					
					\hspace{-0.33cm}{\color{line2}$\diamond$} \,\textbf{Matrix elements at the boundary.}
					
						\vspace{-0.2cm}\begin{siderules}[leftmargin=0.29cm, linecolor=line2]
						
							\vspace{-0.1cm}\begin{itemize}[leftmargin=0.2in]
						
								\item[{\color{mygreen}$\#$}]{\color{mygreen}Base-functions with $d_k^j\neq 0$ have non-vanishing contribution at the bulk, for that reason, we will divide the matrix elements of $M$ and $B$ in two parts: $*^S$ with only contribution in the squares and $*^B$ with only contribution in the triangles at the bulk. (Recall that $F$ has only contribution at the outer boundary of region $\Omega$).}
						
							\end{itemize}
						
						\vspace{0.1cm}\hspace{-0.58cm}{\color{line3}\textbf{for}} $i=N_S+1:2N_S$
						
							\vspace{-0.1cm}\begin{siderules}[leftmargin=0.24cm, linecolor=line3]
							
								\vspace{-0.2cm}\begin{itemize}[leftmargin=0.2in]
								
									\item[{\color{mygreen}$\#$}]{\color{mygreen}Trivial base-functions at the boundary with $d_k^j=\delta_k^j$.}
									
									\vspace{-0.1cm}\item[$\bullet$]Compute piecewise linear base-functions according to Eqs. \eqref{def_base_func} and \eqref{eq:basefunctionbulk} for $d_i^k$ nodes in Fig.\,\ref{sort:bulk} (nodes in the limit of light grey region).
								
								\end{itemize}
							
							\end{siderules}
							
							\vspace{-0.15cm}\hspace{-0.57cm}{\color{line3}\textbf{end}}
							
							\hspace{-0.52cm}{\color{line3}\textbf{if}} $r>N_S$
							
								\vspace{-0.1cm}\begin{siderules}[leftmargin=0.24cm, linecolor=line3]
							
									\vspace{-0.2cm}\begin{itemize}[leftmargin=0.2in]
								
										\item[{\color{mygreen}$\#$}]{\color{mygreen}There are non-trivial base-functions at the boundary with $d_k^j=\delta_k^j$.}
																		
									\end{itemize}
							
								\end{siderules}
								
								\vspace{-0.2cm}\phantom{a}

						\end{siderules}
		\phantom{a}				
		\vspace{-0.7cm}\phantom{a}
		
		\end{siderules}
		
		\newpage
		
		\vspace{0.0cm}\begin{siderules}[leftmargin=0.2cm, linecolor=line1]
	
			\hspace{0.2cm}{\color{line2}$\diamond$}
			
				\vspace{-0.1cm}\begin{siderules}[leftmargin=0.29cm, linecolor=line2]
					
						\vspace{-0cm}\begin{siderules}[leftmargin=0.26cm, linecolor=line3]
						
							\hspace{-0cm}{\color{line4}\textbf{for}} $i=3N_S+1:2N_S+r$
							
								\vspace{-0.1cm}\begin{siderules}[leftmargin=0.24cm, linecolor=line4]
								
									\vspace{-0.1cm}\begin{itemize}[leftmargin=0.225in]
										
										\item[$\bullet$]Compute piecewise linear base-functions.
										
									\end{itemize}
								
								\end{siderules}
								
								\vspace{0.1cm}\hspace{-0.55cm}{\color{line4}\textbf{end}}
						
						\end{siderules}
						
						\vspace{-0.1cm}\hspace{-0.55cm}{\color{line3}\textbf{end}}	
						
							\vspace{-0.2cm}\begin{itemize}[leftmargin=0.2in]
						
								\item[{\color{mygreen}$\#$}]{\color{mygreen}Following, we compute the matrix elements of $M$ and $B$ at the boundary.}
						
							\end{itemize}
						
						\vspace{-0.2cm}\hspace{-0.5cm}{\color{line3}\textbf{for}} $j=1:2N_S+r$
						
							\vspace{-0.1cm}\begin{siderules}[leftmargin=0.26cm, linecolor=line3]
							
								\vspace{-0.1cm}\hspace{-0cm}{\color{line4}\textbf{for}} $i=j:2N_S+r$
							
									\vspace{-0.1cm}\begin{siderules}[leftmargin=0.24cm, linecolor=line4]
									
										\vspace{-0.2cm}\begin{itemize}[leftmargin=0.2in]
						
											\item[{\color{mygreen}$\#$}]{\color{mygreen}Squares contribution.}
						
										\end{itemize}
									
										\vspace{-0.2cm}\begin{itemize}[leftmargin=0.225in]
													
											\item[$\bullet$]Compute $M^S_{i+N_B,\,j+N_B}$ and $B^S_{i+N_B,\,j+N_B}$ with Eqs.~\eqref{matrices_M_B} and \eqref{base_function_square}.
													
										\end{itemize}
										
										\vspace{-0.1cm}\hspace{-0.48cm}{\color{line5}\textbf{if}} $(N_S+1\leq j\leq 2N_S \;\; \text{{\color{line5}\textbf{or}}}\;\;3N_S+1\leq j\leq 2N_S+r)$	
										
											\vspace{-0.15cm}\begin{siderules}[leftmargin=0.24cm, linecolor=line5]
											
												\vspace{-0.15cm}$\text{{\color{line5}\textbf{and}}}\;\;(N_S+1\leq i\leq 2N_S \;\; \text{{\color{line5}\textbf{or}}}\;\;3N_S+1\leq i\leq 2N_S+r)$
												
													\vspace{-0.2cm}\begin{itemize}[leftmargin=0.2in]
						
														\item[{\color{mygreen}$\#$}]{\color{mygreen}Bulk contribution.}
						
													\end{itemize}
												
												\vspace{-0.5cm}\begin{itemize}[leftmargin=0.225in]
													
													\item[$\bullet$]Compute $M^B_{i+N_B,\,j+N_B}$ and $B^B_{i+N_B,\,j+N_B}$ with Eqs. \eqref{Bulk_matrices} and \eqref{Bulk_matrices_sum}.	
													\item[$\bullet$]$M_{i+N_B,\,j+N_B}=M^B_{i+N_B,\,j+N_B}+M^S_{i+N_B,\,j+N_B}$\,;
													\item[$\bullet$]$B_{i+N_B,\,j+N_B}=B^B_{i+N_B,\,j+N_B}+B^S_{i+N_B,\,j+N_B}$\,;
													
												\end{itemize}
											
											\end{siderules}
											
											\vspace{-0.15cm}\hspace{-0.48cm}{\color{line5}\textbf{else}}
										
											\vspace{-0.1cm}\begin{siderules}[leftmargin=0.24cm, linecolor=line5]
												
												\vspace{-0.1cm}\begin{itemize}[leftmargin=0.225in]
													
													\item[$\bullet$]$M_{i+N_B,\,j+N_B}=M^S_{i+N_B,\,j+N_B}$\,;
													\item[$\bullet$]$B_{i+N_B,\,j+N_B}=B^S_{i+N_B,\,j+N_B}$\,;
													
												\end{itemize}
											
											\end{siderules}		
											
											\vspace{-0.15cm}\hspace{-0.48cm}{\color{line5}\textbf{end}}						
								
									\end{siderules}
								
									\vspace{-0.15cm}\hspace{-0.55cm}{\color{line4}\textbf{end}}							
							
							\end{siderules}
							
							\vspace{-0.15cm}\hspace{-0.55cm}{\color{line3}\textbf{end}}	
							
							\vspace{-0.2cm}\begin{itemize}[leftmargin=0.2in]
						
								\item[{\color{mygreen}$\#$}]{\color{mygreen}Computation of the elements of matrix $F$.}
						
							\end{itemize}
							
							\vspace{-0.2cm}\hspace{-0.5cm}{\color{line3}\textbf{for}} $j=2N_S+1:2N_S+r$
						
							\vspace{-0.1cm}\begin{siderules}[leftmargin=0.26cm, linecolor=line3]
							
								\vspace{-0.1cm}\hspace{-0cm}{\color{line4}\textbf{for}} $i=j:2N_S+r$
							
									\vspace{-0.1cm}\begin{siderules}[leftmargin=0.24cm, linecolor=line4]
									
										\vspace{-0.2cm}\begin{itemize}[leftmargin=0.225in]
										
											\item[$\bullet$]Compute $F_{i+N_B,\,j+N_B}$ with Eqs.~\eqref{matrix_F} and \eqref{base_function_square}.
										
										\end{itemize}
									
									\end{siderules}
									
									\vspace{-0.15cm}\hspace{-0.55cm}{\color{line4}\textbf{end}}	
							
							\end{siderules}
							
							\vspace{-0.15cm}\hspace{-0.55cm}{\color{line3}\textbf{end}}	
											
				\end{siderules}
				
				\vspace{-0.15cm}\hspace{-0.55cm}{\color{line2}\textbf{end}}	
				
			\end{siderules}
				
				\newpage
				
				\vspace{0.0cm}\begin{siderules}[leftmargin=0.2cm, linecolor=line1]
	
					\hspace{0.2cm}{\color{line2}$\diamond$} \,\textbf{Cross matrix elements.}
				
					\vspace{-0.1cm}\begin{siderules}[leftmargin=0.29cm, linecolor=line2]
					
						\vspace{-0cm}\begin{itemize}[leftmargin=0.2in]
						
							\item[{\color{mygreen}$\#$}]{\color{mygreen}Cross matrix elements of the first $N_S-8$ base-functions (which are sorted in counterclockwise order in Fig.\,\ref{sort:bulk}) with base-functions at the boundary with $d_k^j=\delta_k^j$.}
						
						\end{itemize}
								
					\vspace{0.1cm}\hspace{-0.5cm}{\color{line3}\textbf{for}} $j=1:N_S-8$
							
						\vspace{-0.1cm}\begin{siderules}[leftmargin=0.24cm, linecolor=line3]
						
							\vspace{-0.1cm}\hspace{-0cm}{\color{line4}\textbf{for}} $i=1:2N_S+r$
							
								\vspace{-0.1cm}\begin{siderules}[leftmargin=0.24cm, linecolor=line4]
								
									\vspace{-0.1cm}\hspace{0.1cm}{\color{line5}\textbf{if}} $(N_S+1\leq i\leq 2N_S \;\; \text{{\color{line5}\textbf{or}}}\;\;3N_S+1\leq i\leq 2N_S+r)$	
										
											\vspace{-0.15cm}\begin{siderules}[leftmargin=0.26cm, linecolor=line5]
											
												\vspace{-0cm}\begin{itemize}[leftmargin=0.225in]
												
													\item[$\bullet$]Compute $M_{i+N_B,\,j}$ and $B_{i+N_B,\,j}$ with Eqs.~\eqref{Bulk_matrices} and \eqref{Bulk_matrices_sum}.
												
												\end{itemize}
											
											\end{siderules}
											
											\vspace{-0.15cm}\hspace{-0.48cm}{\color{line5}\textbf{end}}
								
								\end{siderules}
								
								\vspace{-0.15cm}\hspace{-0.55cm}{\color{line4}\textbf{end}}
								
						\end{siderules}
								
						\vspace{-0.15cm}\hspace{-0.55cm}{\color{line3}\textbf{end}}
				
				\end{siderules}
				
				\vspace{-0.15cm}\hspace{-0.55cm}{\color{line2}\textbf{end}}
				
				\hspace{-0.33cm}{\color{line2}$\diamond$} \,\textbf{Matrix elements in the upper triangular part.}
				
					\vspace{-0.15cm}\begin{siderules}[leftmargin=0.29cm, linecolor=line2]		
					
						\vspace{-0.1cm}\begin{itemize}[leftmargin=0.2in]
						
							\item[{\color{mygreen}$\#$}]{\color{mygreen}Fill the upper triangular part of matrices $M$, $B$ and $F$ with the elements below the main diagonal.}
						
						\end{itemize}
				
						\vspace{0.1cm}\hspace{-0.5cm}{\color{line3}\textbf{for}} $j=1:N_B+2N_S+r-1$
				
							\vspace{-0.1cm}\begin{siderules}[leftmargin=0.26cm, linecolor=line3]
					
								\vspace{-0.1cm}\hspace{-0cm}{\color{line4}\textbf{for}} $i=j+1:N_B+2N_S+r-1$
						
									\vspace{-0.1cm}\begin{siderules}[leftmargin=0.26cm, linecolor=line4]
							
										\vspace{-0cm}\begin{itemize}[leftmargin=0.225in]
								
											\item[$\bullet$]$M_{ji}=M_{ij}$\,;\; $B_{ji}=B_{ij}$\,;\; $F_{ji}=F_{ij}$\,;
								
										\end{itemize}
							
								\end{siderules}
							
								\vspace{-0.15cm}\hspace{-0.55cm}{\color{line4}\textbf{end}}
											
							\end{siderules}
					
					\vspace{-0.15cm}\hspace{-0.55cm}{\color{line3}\textbf{end}}
					
					\end{siderules}
					
					\vspace{-0.15cm}\hspace{-0.55cm}{\color{line2}\textbf{end}}
					
					\hspace{-0.33cm}{\color{line2}$\diamond$} \,\textbf{Spectral problem.}
					
						\vspace{-0.15cm}\begin{siderules}[leftmargin=0.29cm, linecolor=line2]	
						
							\vspace{-0.1cm}\begin{itemize}[leftmargin=0.225in]
							
								\item[$\bullet$]Solve $(M-F)u_i=\lambda_i Bu_i$.
							
							\end{itemize}
						
						\end{siderules}
						
						\vspace{-0.15cm}\hspace{-0.55cm}{\color{line2}\textbf{end}}
										
		\end{siderules}
		
\vspace{-0.15cm}\noindent{\color{line1}\textbf{end}}


{\footnotesize

\begin{thebibliography}{99}

\bibitem{marsden01}
{\sc R.~Abraham, J.~Marsden, and T.~Ratiu}.
\newblock {\em Manifolds, Tensor Analysis and Applications}.
\newblock Springer, 1988.

\bibitem{adams03}
{\sc R.A. Adams and J.J.F. Fournier}.
\newblock {\em {S}obolev Spaces}.
\newblock Pure and Applied Mathematics. Academic Press, Oxford, second edition,
  2003.
  
\bibitem{asorey06}
{\sc M. Asorey. G.G. Alvarez, J.M. Mu\~noz-Casta\~neda}
\newblock{\em Casimir effect and Global theory of boundary conditions.}
J. Phys. A: Math. and General \textbf{39} (21), 6127 (2006).

\bibitem{asorey13}
{\sc M.~Asorey, A.P.~Balachandran and J.M.~P\'{e}rez-Pardo}. 
{\it Edge states: Topological insulators, superconductors and QCD chiral bags.} 
JHEP {\bf 2013}, (12) 073 (2013).

\bibitem{asorey16}
{\sc M.~Asorey, A.P.~Balachandran and J.M.~P\'{e}rez-Pardo}. 
{\it Edge states at phase boundaries and their stability}.
Reviews in Mathematical Physics \textbf{28} (9), 1650020 (2016).

\bibitem{As83} 
{\sc M. Asorey, J.G. Esteve, A.F. Pacheco.}
{\it Planar rotor: The $\theta$-vacuum structure, and some approximate methods in quantum mechanics.} 
Phys. Rev. D, {\bf 27}, 1852-68 (1982).

\bibitem{As15} 
{\sc M. Asorey, A. Ibort, G. Marmo}.
\textit{The topology and geometry of self-adjoint and elliptic boundary conditions for
Dirac and Laplace operators}.
Int. J. Geom. Methods in Modern Physics, \textbf{12} (6) 1561007 (2015).

\bibitem{babuska07}
{\sc I.~Babu\u{s}ka, F.~Nobile and R.~Tempone}
\newblock {\em A stochastic Collocation method for elliptic partial differential equations with random input data}
SIAM J. Num. Anal., \textbf{45} (3), 1005--1034 (2007).

\bibitem{babuska91}
{\sc I.~Babu\u{s}ka, J.~Osborn}
\newblock{\em Eigenvalue Problems}
Handbook of numerical analysis, vol. II
\newblock Elsevier, North-Holland, 1991.

\bibitem{babuska92}
{\sc I.~Babu\u{s}ka and M.~Suri}.
\newblock{\em On locking and Robustness in the Finite Element Method}.
SIAM J. Num. Anal., \textbf{29} 1261--1293 (1992).

\bibitem{babuska04}
{\sc I.~Babu\u{s}ka, R.~Tempone and G.E.~Zouraris}.
\newblock{\em Galerkin finite Element Approximations of Stochastic Elliptic Partial Differential Equations}.
SIAM J. Numer. Anal. \textbf{42}, 800--825 (2004).

\bibitem{babuska05}
{\sc I.~Babu\u{s}ka, R.~Tempone and G.E.~Zouraris}.
\newblock{\em Solving elliptic boundary value problems with uncertain coefficients by the finite element method- the stochastic formulation.}
Compt. Methods Appl. Mech. Engng. \textbf{194}, 1251--1294 (2005).

\bibitem{BZ06}
{\sc B.A.\ Bernevig and S.-C.\ Zhang.}
\newblock{\em Quantum Spin Hall Effect}
PRL \textbf{96} (10) 106802 (2006)

\bibitem{boffi10}
{\sc D.~Boffi}.
\newblock{\em Finite element approximation of eigenvalue problems}.
Acta Numerica \textbf{19}, 1--120 (2010).

\bibitem{brenner08}
{\sc S.C. Brenner and L.R. Scott}.
\newblock {\em The Mathematical Theory of Finite Element Methods}.
\newblock Springer, second edition, 2008.

\bibitem{CP17}
{\sc R.\ Carlone, A.\ Posilicano.}
\emph{A quantum hybrid with a thin antenna at the vertex of a wedge.}
Phys. Lett. A {\bf 381}(12), 1076-1080 (2017).

\bibitem{cw08}
{\sc J.\ Clarke and F.K. Wilhelm}
\emph{Superconducting quantum bits}.
\newblock Nature \textbf{453}, 1031--1042 (2008).

\bibitem{SAHR11}
{\sc S.\ Das Sarma, S.\ Adam, E.H.\ Hwang and E. Rossi.}
\emph{Electronic transport in two dimensional graphene.}
Rev. Mod. Phys. {\bf 83}, 407-470 (2011).


\bibitem{davies95}
{\sc E.B. Davies}.
\newblock {\em Spectral Theory and Differential Operators}.
\newblock Cambridge University Press, 1995.

\bibitem{DAFT08}
{\sc Dell'Antonio G.F., Figari R., Teta A.} 
{\it A brief review on point interactions in Inverse Problems and Imaging}. 
Lectures Notes in Mathematics, Springer (2008).

\bibitem{demmel97}
{\sc J.W. Demmel}.
\newblock {\em Applied Numerical Linear Algebra}.
\newblock SIAM, 1997.

\bibitem{ES97}
{\sc P.\ Exner and P.\ \v{S}eba.}
\emph{Resonance statistics in a microwave cavity with a thin antenna.}
Phys.\ Lett.\ A {\bf 228}(3), 146-150 (1997).

\bibitem{ES07}
{\sc P.\ Exner and P.\ \v{S}eba.}
\emph{A ``hybrid plane'' with spin-orbit interaction.}
Russ. J. Math. Phys. {\bf14}(4), 430-434 (2007). 


\bibitem{Grubb68}
{\sc G.~Grubb}. \emph{A characterization of the nonlocal boundary value problems associated with an elliptic operator}.
Ann. Sc. Norm. Sup., {\bf 22}(3), 425--513 (1968).

\bibitem{grubb12}
{\sc G.~Grubb}
\emph{Krein-like extensions and the lower boundedness problem for elliptic operators}.
J. Differential Equations {\bf 252}, 852--885 (2012).

\bibitem{HO12}
{\sc D.W.~Hahn and M.N.~}
\emph{Heat Conduction.}
\newblock John Wiley \& Sons, (2012).

\bibitem{Ib14} 
{\sc A. Ibort, F. Lled\'o, J.M. P\'erez-Pardo}.
\textit{On self-adjoint extensions and symmetries in quantum mechanics}. 
Annales Inst. Henri Poincar\'e, \textbf{16},  2367--2397 (2015).


\bibitem{ibortlledo14b}
{\sc A.~Ibort, F.~Lled\'{o} and J.M. P\'{e}rez-Pardo}.
\textit{Self-adjoint extensions of the Laplace-Beltrami operator and unitaries at the boundary}.  
J. Funct. Analysis, \textbf{268} (3), 634--670 (2015).
  
\bibitem{IMP} 
{\sc A. Ibort, G. Marmo, J.M. P\'erez--Pardo}.
\textit{Boundary dynamics driven entanglement}. 
Journal of Physics A: Mathematical and Theoretical, \textbf{47} (38), 385301 (2014).

\bibitem{Ib13}  
{\sc A. Ibort, J.M. P\'erez-Pardo}.
\textit{Numerical solutions of the spectral problem for arbitrary self-adjoint extensions of 1D Schr\"odinger equation}. 
SIAM J. Num. Anal., \textbf{51} (2) 1254--1279 (2013).


\bibitem{Ib15} 
{\sc A. Ibort, J.M. P\'erez-Pardo}.
\textit{On the theory of self-adjoint extensions of symmetric operators and its applications to Quantum Physics}.
Int. J. Geom. Methods in Modern Physics, \textbf{12} (6), 1560005 (2015).

\bibitem{kato95}
{\sc T.~Kato}.
\newblock {\em Perturbation Theory for Linear Operators}.
\newblock Classics in Mathematics. Springer, 1995.

\bibitem{kittel96}
{\sc C.~Kittel.}
\newblock \emph{Introduction to Solid State Physics.}
John Wiley \& Sons, New York, 1996.

\bibitem{lions72}
{\sc J.L. Lions and E.~Magenes}.
\newblock {\em Non-homogeneous Boundary Value Problems and Applications},
  volume~I of {\em Grundlehren der mathematischen Wissenschaften}.
\newblock Springer, 1972.

\bibitem{marolf97}
{\sc D.~Marolf}.
\newblock {\em Interpolating between topologies: Casimir energies}
\newblock {Phys.\ Lett.\ B} {\bf 392}, 287 (1997).

\bibitem{mar05}
{\sc J.M.\ Martinis et al.}
\newblock {Decoherence in Josephson Qubits from Dielectric Loss.}
\newblock Phys. Rev. Letters \textbf{95}, 210503 (2005).

\bibitem{morandi88}
{\sc G.~Morandi}.
\newblock {\it Quantum Hall effect: Topological problems in condensed-matter
  physics}.
\newblock Monographs and textbooks in physical science. Lecture Notes 10.
  Bibliopolis, Naples, 1988.

\bibitem{mohsen82}
{\sc A.~Mohsen.}
\newblock {\em On the impedance boundary conditions.}
\newblock Applied Mathematical Modelling \textbf{6} (5), 405--407 (1082).

\bibitem{norton08}
{\sc R.~Norton.}
\emph{Numerical computation of band gaps in photonic crystal fibres.}
\newblock PhD Thesis. University of Bath, 2008.

\bibitem{Pe17}
{\textsc{J.M.\ P\'erez-Pardo}.}
{\textit{Dirac-like operators on the Hilbert space of differential forms on manifolds with boundaries}.}
Int.\ J.\ Geom.\ Methods Mod.\ Phys {\textbf{14}} (2017).

\bibitem{Pe15} 
{\sc J.M. P\'erez-Pardo, M. Barbero-Li\~n\'an, A. Ibort}.
\textit{Boundary dynamics and topology change in quantum mechanics} .
Int. J. Geom. Methods in Modern Physics, \textbf{12}, 1560011 (2015).

\bibitem{plunien86}
{\sc G.~Plunien, B.~M\"{u}ller, and W.~Greiner}.
\newblock The {C}asimir effect.
\newblock {\it Phys. Rep.}, {\bf 134}, 87--193  (1986).

\bibitem{ram02}
{\sc R.~Ram-Mohan}. 
\newblock{\em Finite Element and Boundary Element Applications in Quantum Mechanics}.
\newblock Oxford University Press, 2002.

\bibitem{reed72}
{\sc M.~Reed and B.~Simon}.
\newblock {\em Methods of Modern Mathematical Physics}, volume~I.
\newblock Academic Press, New York, 1972.

\bibitem{reed75}
{\sc M.~Reed and B.~Simon}.
\newblock {\em Methods of Modern Mathematical Physics}, volume~II.
\newblock Academic Press, New York, 1975.

\bibitem{reed78}
{\sc M.~Reed and B.~Simon}.
\newblock {\em Methods of Modern Mathematical Physics}, volume~IV.
\newblock Academic Press, New York, 1978.

\bibitem{schmuedgen12}
{\sc K.~Schm\"{u}dgen}.
\newblock {\em {U}nbounded self-adjoint operators on {H}ilbert space}.
\newblock Springer, Dordrecht, 2012.

\bibitem{strouboulis00}
{\sc T.~Strouboulis, I.~Babuska and K.~Copps}.
\newblock{\em The design and analysis of the Generalised Finite Element Method}.
Compt. Methods Appl. Mech. Engng. \textbf{181}, 43--69 (2000).

\bibitem{sukumar09}
{\sc N.~Sukumar and J.E.~Pask}.
\newblock{\em Classical and enriched finite element formulations for Bloch-periodic boundary conditions}.
Int. J. Numer. Meth. Engng. \textbf{77}, 1121--1138 (2009).

\bibitem{taylor96}
{\sc M.E. Taylor}.
\newblock {\em Partial Differential Equations I}, volume 115 of {\em Applied Mathematical Sciences}.
\newblock Springer, 1996.

\end{thebibliography}
}
\end{document}